\documentclass[leqno,12pt, draft]{amsart} %
\usepackage{amssymb} 

\usepackage{amsmath}
\usepackage{amsbsy}
\usepackage{amscd}
\usepackage{latexsym}
\usepackage{amscd}
\usepackage[all]{xy}
\usepackage{amsthm}
\usepackage{mathrsfs}

\usepackage[OT2,T1]{fontenc}
\DeclareSymbolFont{cyrletters}{OT2}{wncyr}{m}{n}
\DeclareMathSymbol{\Sha}{\mathalpha}{cyrletters}{"58}

\DeclareFontEncoding{OT2}{}{}
\DeclareFontSubstitution{OT2}{cmr}{m}{n}

\DeclareFontFamily{OT2}{cmr}{\hyphenchar\font45 }
\DeclareFontShape{OT2}{cmr}{m}{n}{%
   <5><6><7><8><9>gen*wncyr%
   <10><10.95><12><14.4><17.28><20.74><24.88>wncyr10}{}
\DeclareFontShape{OT2}{cmr}{b}{n}{%
   <5><6><7><8><9>gen*wncyb%
   <10><10.95><12><14.4><17.28><20.74><24.88>wncyb10}{}

\DeclareMathAlphabet{\mathcyr}{OT2}{cmr}{m}{n}
\DeclareMathAlphabet{\mathcyb}{OT2}{cmr}{b}{n}
\SetMathAlphabet{\mathcyr}{bold}{OT2}{cmr}{b}{n}

\theoremstyle{plain} 
\newtheorem{thm}{\indent\sc Theorem}[section] 
\newtheorem{lem}[thm]{\indent\sc Lemma}
\newtheorem{cor}[thm]{\indent\sc Corollary}
\newtheorem{prop}[thm]{\indent\sc Proposition}

\newtheorem{conj}[thm]{\indent\sc Conjecture}

\theoremstyle{definition} %
\newtheorem{dfn}[thm]{\indent\sc Definition}
\newtheorem{rem}[thm]{\indent\sc Remark}

\newtheorem{assume}[thm]{\indent\sc Assumption}

\newtheorem*{thm'}{\indent\sc Theorem'}
\renewcommand{\theequation}{\thesection.\arabic{equation}}

\newcommand{\Zp}{\mathbb{Z}_p}
\newcommand{\Qp}{\mathbb{Q}_p}
\newcommand{\Hom}{\mathrm{Hom}}

\newcommand{\Z}{\mathbb{Z}}
\newcommand{\Q}{\mathbb{Q}}

\newcommand{\Fr}{\mathrm{Fr}}

\newcommand{\R}{\mathbb{R}}

\newcommand{\Oc}{\mathscr{O}}

\newcommand{\C}{\mathbb{C}}
\newcommand{\Gal}{\mathrm{Gal}}

\newcommand{\Sel}{\mathrm{Sel}}

\newcommand{\ord}{\mathrm{ord}}
\newcommand{\z}{\mathfrak{z}}

\newcommand{\tors}{\mathrm{tors}}

\newcommand{\ur}{\mathrm{ur}}

\newcommand{\loc}{\mathrm{loc}}
\newcommand{\ka}{\kappa}

\newcommand{\Supp}{\mathrm{Supp}}

\begin{document}

\title[Refined BSD conjecture]{Kato's Euler system and\\ the Mazur-Tate refined conjecture of BSD type} 

\author[K.\ Ota]{Kazuto Ota} 

\address{
Mathematical Institute, Tohoku University,  Sendai 980-8578, Japan
}
\email{sb0m07@math.tohoku.ac.jp}

\maketitle

\begin{abstract}Mazur and Tate proposed a conjecture which compares the Mordell-Weil rank of an elliptic curve over $\Q$ with the order of vanishing of Mazur-Tate elements, which are analogues of Stickelberger elements.
Under some relatively mild assumptions, we prove this conjecture.
Our strategy of the proof is  to study divisibility of certain derivatives of Kato's Euler system.
\end{abstract}

\tableofcontents

\makeatletter
  \renewcommand{\theequation}{%
         \thesection.\arabic{equation}}
  \@addtoreset{equation}{section}
\makeatother

\section{Introduction}\label{intro}

\subsection{The Mazur-Tate refined conjecture of BSD type}
 Mazur-Tate \cite{m-t} proposed a \textit{refined conjecture of BSD type}, which predicts mysterious relations between arithmetic invariants of an elliptic curve $E$ over $\Q$ and Mazur-Tate elements constructed from modular symbols.
A Mazur-Tate element is an analogue of Stickelberger element and refines the $p$-adic $L$-function of $E$.
As the Birch and Swinnerton-Dyer conjecture, the Mazur-Tate refined conjecture of BSD type consists of two parts.
 One compares the Mordell-Weil rank with the ``order of vanishing'' of Mazur-Tate elements (the rank-part).
 The other describes the ``leading coefficients'' of the elements.
The aim of this paper is to prove the rank-part under some mild assumptions.  
Here, we explain this part more precisely (see Section \ref{chapter of refined BSD} for the other part).

For a positive integer $S$, we put $G_S=\Gal(\Q(\mu_S)/\Q).$ 
The \textit{Mazur-Tate element} $\theta_S$ is an element of $ \Q[G_S]$ such that 
for every character $\chi$ of $G_S$, the evaluation  
$\chi(\theta_S)$ equals the algebraic part of  $L(E,\chi^{-1},1)$ up to an explicit factor.
It is important that the denominators of $\theta_S$ are bounded as $S$ varies, which implies the existence of non-trivial congruences between these special values as $\chi$ varies.
If $E$ is a strong Weil curve, then $\theta_S \in \Z[1/tc(E)][G_S]$, where $t:=|E(\Q)_{\tors}|$, and $c(E) \in \Z$ denotes the Manin constant, which is conjectured to be $1$ in this case.

 Let $R$ be a subring of $\Q$ such that $\theta_S \in R[G_S]$ for all $S>0$.
We denote by $I_S$ the augmentation ideal of $R[G_S]$ and by $\mathrm{sp}(S)$ the number of split multiplicative primes of $E$ dividing $S$. 
We put $r_E=\mathrm{rank}(E(\Q))$.
The following is the rank-part of the Mazur-Tate refined conjecture of BSD type.
 \begin{conj}[Mazur-Tate]\label{main conj intro}The order of vanishing of $\theta_S$ at the trivial character is greater than or equal to $r_E+\mathrm{sp}(S),$ that is, 
 \begin{equation*}
  \theta_S \in I_S^{r_E+\mathrm{sp}(S)}.
    \end{equation*}
  \end{conj}

\subsection{The main result}\label{subsection: main result}
We suppose that $E$ does not have complex multiplication, and we denote by $N$ the conductor.
 Let $R$ be a subring of $\Q$ in which all the  primes $p$ satisfying at least one of the following conditions are invertible:
\begin{enumerate}
\renewcommand{\labelenumi}{(\roman{enumi})}
\item $p$ divides $6N\cdot|E(\mathbb{F}_p)|\prod_{\ell|N}[E(\Q_{\ell}):E_0(\Q_{\ell})]$, where for a prime $\ell$,  we denote by $E_0(\Q_{\ell})$ the group of points in $E(\Q_{\ell})$ whose reduction is a non-singular point of $E(\mathbb{F}_{\ell})$,
\item the Galois representation of $G_{\Q}$ attached to the $p$-adic Tate module is \textit{not} surjective,
\item $p< r_E$.
\end{enumerate}
The following is our main result.
\begin{thm}[Theorem \ref{weak vanishing thm}]\label{main}
Let $S$ be a square-free product of good primes $\ell$ such that for each prime $p$ not invertible in $R$,  the module $E(\mathbb{F}_{\ell})[p]$ is cyclic, that is, $E(\mathbb{F}_{\ell})[p]$ is isomorphic to $\Z/p\Z$ or $0$.
Then, $\theta_S \in R[G_S]$, and Conjecture \ref{main conj intro} holds, that is,
\begin{equation*}
\theta_S \in I_S^{r_E}.
\end{equation*}
\end{thm}
\begin{rem}\label{first remark}
\begin{enumerate}
\item The density (if it exists)  of primes $\ell$ that satisfy the assumption of Theorem \ref{main} is greater than $0.99$ (see Remark \ref{density computation} for the detail).
 We also note that each good supersingular prime $\ell $ of $E$ satisfies the assumption of Theorem \ref{main}.
\item We mention known results on Conjecture \ref{main conj intro}.
When $p$ is a good ordinary prime, Kato's result (\cite{kat}) on the $p$-adic BSD conjecture proves that $\theta_{p^n} \in \Zp\otimes I_{p^n}^{r_E}$ for $n\ge 0$.
Kurihara's result in \cite{kur14} implies that 
$\theta_S \in \Z_{p}\otimes I_S^{r_E}$ where $S$ does not need to be a power of $p$  (cf.\  \cite[Remark 2]{kur14} and \cite[Proposition 3]{m-t}).
However he was assuming the $\mu=0$ conjecture.
For a supersingular prime $p$, in their unpublished work, Emerton, Pollack and Weston seem to have proved a similar assertion at least when $S$ is a power of $p$. 
Tan \cite{tan} proved Conjecture \ref{main conj intro} for many $S$ without extending the scalar to $\Zp$. 
However, he was assuming the \textit{full} BSD conjecture not only over $\Q$ but also over cyclic extensions $K$ of $\Q$ inside $\Q(\mu_S).$
Note that Theorem \ref{main} does not require the validity of any conjecture.
\item It may happen that $\theta_S$ has an extra zero, that is, $\theta_S \in  I_S^{r_E+\mathrm{sp}(S)+1}$ (see Remark \ref{remark on the conj} and Theorem \ref{p-adic trivial zeros}).
\item 
By Serre \cite{ser72}, there are only finitely many primes satisfying (ii).
By \cite[Lemma 8.18]{maz72}, if $E(\Q)$ has a non-trivial torsion point
then there are at most three good primes $p$ dividing $|E(\mathbb{F}_p)|$.
\item  The assertion that $\theta_S \in R[G_S]$ is due to \cite[\S 3]{ste} and  (ii).
\end{enumerate}
\end{rem}
By \cite[Theorem 2]{coj} and \cite[Th\'eor\`em $4^{\prime}$]{ser72}, we have the following.
\begin{cor}\label{kind example}We assume that $E(\Q)$ has a non-trivial torsion point.
We put 
\begin{equation*}
d=\max\left\{r_E,\ \frac{4\sqrt{6}}{3}N\prod_{\ell|N}\left(1+\frac{1}{\ell}\right)^{\frac{1}{2}}+1 \right\}\ \ \ \text{and}\ \ R=\Z[p^{-1}; p \ \text{is a prime less than}\  d]. 
\end{equation*}
Then, for every square-free product $S$ of good supersingular primes, $\theta_S \in R[G_S]$ and 
\begin{equation*}
\theta_S \in I_S^{r_E}.
\end{equation*}
\end{cor}

In this paper, we also give a partial evidence (Theorem \ref{theorem on leading term}) of the part of the Mazur-Tate refined conjecture which relates arithmetic invariants such as the Tate-Shafarevich group $\Sha$ to the \textit{leading coefficient} of $\theta_S$ defined as 
 the image $\tilde{\theta}_S$ of $\theta_S$ in $I_S^{r_E}/I_S^{r_E+1}$. 
The following is a special case of Theorem \ref{theorem on leading term}.
\begin{thm}\label{leading term thm in intro} Let $p$ be a prime not invertible in $R$ and  $S$  a square-free product of good primes $\ell$ such that $\ell \equiv 1 \bmod p$ and $E(\mathbb{F}_{\ell})[p]$ is cyclic.  
If $\tilde{\theta}_S \not\equiv 0 \bmod p(I_S^{r_E}/I_S^{r_E+1})$, then 
\begin{center}
 $\Sha[p]=0$\ \ and \ \ $p\nmid J_S$,
\end{center}
where  $J_S$ is the order of the cokernel of the map $E(\Q) \to \left(\oplus_{\ell |S} E(\mathbb{F}_{\ell})\right)\oplus\left(\oplus_{\ell |N}E(\Q_{\ell})/E_0(\Q_{\ell})\right)$. 
\end{thm}

\subsection{The plan of proof}\label{subsection: the proof}
We briefly explain how to prove Theorem \ref{main}.
 By a group ring theoretic argument (Lemma \ref{globallocal}), we are reduced to proving that
\begin{equation}\label{final goal}
\theta_S \in \Zp\otimes_R I_S^{r_E} \ \ \text{for all primes} \ p\ \text{not invertible in}\ R.
\end{equation}
Let $p$ be a such prime and
  $r_{p^{\infty}}$ the $\Zp$-corank of the (discrete) Selmer group $\Sel(\Q,E[p^\infty])$.

Our strategy of the proof of (\ref{final goal}) is to show that \textit{Darmon-Kolyvagin derivatives} of Kato's Euler system are divisible by a power of $p$.
In order to investigate the  \textit{divisibility}, we modify an argument of Darmon \cite{dar92}, who proposed a refined conjecture for Heegner points and proved an analogue of Conjecture \ref{main conj intro} in many cases (see \cite{l-v} for a generalization to Heegner cycles).
Next, by modifying ideas of Kurihara \cite{kur02}, Kobayashi \cite{kob}  and Otsuki \cite{ots},
we relate Kato's Euler system with Mazur-Tate elements.
Then, the derivatives of Kato's Euler system appear in the coefficients of the \textit{Taylor expansion} of $\theta_S$.
By the divisibility of the derivatives and a group-ring theoretic argument, we show that $\theta_S$ belongs to a power of the augmentation ideal.
However, our modification of Darmon's argument implies only that 
\begin{equation}\label{-1 gap}
\theta_S \in \Zp \otimes_R I_S^{\min\{r_{p^{\infty}}-1, p \}}.
\end{equation}
One might expect that  Darmon's argument  implies that $\theta_S \in \Zp \otimes_R I_S^{\min\{r_{p^{\infty}}, p \}}.$
The obstruction is the difference between the local condition at $p$ of Heegner points and that of Kato's Euler system.
The localization of Heegner points at $p$ obviously comes from local rational points (i.e. it is crystalline at $p$), and then Heegner points are related to the usual Selmer group.
However,  the localization of Kato's Euler system does not necessarily come from a local rational point, and then
we can relate Kato's Euler system only with the strict Selmer group $H^1_{f,p}(\Q,E[p^\infty])$, whose local condition  at $p$ is zero.
Since the corank of $H^1_{f,p}(\Q,E[p^\infty])$ is not necessarily greater than $r_{p^\infty}-1,$ we have only (\ref{-1 gap}).

Our idea for deducing (\ref{final goal}) from (\ref{-1 gap}) is to  apply the $p$-parity conjecture, which is now a theorem (cf.\ \cite{dok}, \cite{kim}, \cite{nek}). 
It asserts that $r_{p^{\infty}} \equiv \mathrm{ord}_{s=1}(L(E,s)) \bmod 2$.
On the other hand, the functional equation of $\theta_S$ implies that if $\theta_S \in (\Zp\otimes I_S^b) \setminus ( \Zp\otimes I_S^{b+1})$ for some $b >0$ then $b \equiv \mathrm{ord}_{s=1}L(E,s) \bmod 2.$
Combining these congruences with (\ref{-1 gap}), we deduce (\ref{final goal}).

\begin{rem}Divisibility of derivatives of Euler systems plays an important role in the proof of not only Theorem \ref{main} but also 
other theorems.
By investigating such divisibility, in this paper, we also show Theorems \ref{leading term thm in intro} and  \ref{sel}. The latter theorem gives a construction of $\Q$-rational points of $E$ (modulo $p$) from certain indivisibility of Euler systems.
\end{rem}

\textbf{Notation.}
Throughout this paper, let $E$ be an elliptic curve over $\Q$ of conductor $N$ without complex multiplication. We put $r_E=\mathrm{rank}(E(\Q))$ and $m_{\ell}=[E(\Q_{\ell}):E_0(\Q_{\ell})].$

For an abelian group $M$ and an integer $n$, we write $M/n=M/nM$.
We denote  by $M_{\tors}$ the maximal torsion subgroup of $M$.
 
For a field $K$, 
we denote by $G_K$ the absolute Galois group $\Gal(\overline{K}/K)$, where $\overline{K}$ is a separable closure of $K$.
We fix embeddings $\overline{\Q} \hookrightarrow \mathbb{C}$ and $\overline{\Q} \hookrightarrow \overline{\Q}_p$ for every prime $p$. 
For an integer $S$, we put $\zeta_S = \mathrm{exp}(2
\pi i/S)$ and  $G_S=\Gal(\Q(\zeta_S)/\Q).$


\textbf{Acknowledgements.}
This paper is based on the author's thesis. He would like to express his sincere gratitude to his advisor Professor Shinichi Kobayashi for his insightful advice and discussion.
Part of this work was completed while the author was visiting l'Institute de Math\'emetiques de Jussieu with a support by the JSPS Strategic Young Researcher Overseas Visits Program for Accelerating Brain Circulation.
He is deeply grateful to Professor Jan Nekov\'a\v r for his hospitality.
Thanks are due to  Professor Masato Kurihara for informing the author of his work related to the Mazur-Tate refined conjecture.
The author would like to thank Matteo Longo for discussion, and Chan-Ho Kim for showing his note of talks by Robert Pollack.
This work was supported by Grant-in-Aid for JSPS Fellows 12J04338.

\section{Mazur-Tate elements}\label{chapter of refined BSD}
In this section, we recall the definition of Mazur-Tate elements, and we briefly review the Mazur-Tate refined conjecture of BSD type in a simple case. 

 We fix a global minimal Weierstrass model of $E$ over $\Z$ and the N\'eron differential $\omega$.
  Then, we have a natural map from the first homology group $H_1(E(\C),\Z)$ to $\C$ 
\begin{equation*}
H_1(E(\C),\Z) \to \C; \ \ \ \gamma \mapsto \int_{\gamma}\omega.
\end{equation*}We denote by $\Lambda$ the image of this map.
Let $\Omega^{+}, -i \Omega^{-}>0$ be the largest  numbers  such that 
\begin{equation*}
\Lambda \subseteq \Z\Omega^{+} \oplus \Z  \Omega^{-}.
\end{equation*} By \cite{wil}, \cite{t-w} and \cite{b-c-d-t}, let $f(z)=\sum_{n\ge 1}a_n \exp(2\pi i nz)$ be the  newform corresponding to $E$.
Let $L(E,s)=\sum_{n\ge 1} a_nn^{-s}$ be the Hasse-Weil $L$-function of $E$.
For a Dirichlet  character $\chi$, we put  $L(E,\chi,s)= \sum_{n\ge 1} \chi(n)a_nn^{-s}$.
For integers $a$ and $S$ with $S>0$, we define $\left[a/S\right]_E^{\pm} \in \R$ by
\begin{equation*}
2\pi\int_0^{\infty}f\left(\frac{a}{S}+it\right)dt =\left[\frac{a}{S}\right]_E^+\Omega^{+} +\left[\frac{a}{S}\right]_E^-\Omega^{-}.
\end{equation*}
The Manin-Drinfeld theorem (\cite{dri}, \cite{man}) implies that
  $\left[a/S\right]_E^{\pm}\in \Q.$
In the terminology of \cite{m-t-t},
\begin{equation}\label{known terminology}
\lambda(f,1;-a,S)=\left[\frac{a}{S}\right]_E^+\Omega^{+} +\left[\frac{a}{S}\right]_E^-\Omega^{-}.
\end{equation}

\begin{dfn}\label{def of mazur-tate}For a positive integer $S$,
we define an element $\theta_S$ of $\Q[G_S]$ by
\begin{equation*}
\theta_S=\sum_{a\in(\Z/S\Z)^{\times}}\left( \left[\frac{a}{S}\right]_E^++\left[\frac{a}{S}\right]_E^-\right)\delta_a\in\Q[G_S],
\end{equation*}where $\delta_a\in G_S$ is the element satisfying $\delta_a\zeta_S=\zeta_S^a$. 
We call $\theta_S$ the \textit{Mazur-Tate element}.
\end{dfn}
\begin{rem}\label{different notation}
Our $\theta_S$ slightly differs from the original Mazur-Tate element, which is called the \textit{modular element} in \cite{m-t}.
The image of $\frac{1}{2}\theta_S$ in $\Q[\Gal(\Q(\mu_S)^{+}/\Q)]$ coincides with their modular element,
where $\Q(\mu_S)^{+}$ is the maximal totally real subfield of $\Q(\mu_S)$.
\end{rem}
For $n|m$, we denote by $\pi_{m/n}$ the map $\Q[G_m] \to \Q[G_n]$ induced by the natural surjection $G_m\to G_n.$
For a character $\chi $ of $G_S$, we put 
\begin{equation*}
\tau_S(\chi)=\sum_{\gamma\in G_S} \chi(\gamma)\zeta_S^{\gamma}.
\end{equation*}
In this paper, square-free integers $S$ are of particular interests.
By (\ref{known terminology}) and \cite[Chapter 1, \S 4 and \S 8]{m-t-t}, we have the following.
\begin{prop}\label{mazur-tate}
\begin{enumerate} \item Let $S$ be a positive integer and $\ell$ a prime not dividing $S$.
Then,
\begin{align*}
\pi_{S\ell/S}\theta_{S\ell}&= -\Fr_{\ell}(1-a_{\ell}\Fr_{\ell}^{-1}+\epsilon(\ell)\Fr_{\ell}^{-2})\theta_S,
\end{align*}where $\epsilon$ is the trivial Dirichlet character modulo $N$, and $\Fr_{\ell}$ denotes the arithmetic Frobenius of $\ell$.
\item For a character $\chi$ of $G_S$ with conductor $S$,
we have
\begin{equation*}
\chi(\theta_S)=\tau_S(\chi)\frac{L(E,\chi^{-1},1)}{\Omega^{\chi(-1)}}.
\end{equation*}
\end{enumerate}
\end{prop}

We briefly review the Mazur-Tate refined conjecture of BSD type.
Let $S$ be a square-free positive integer and $R$ a subring 
 of $\Q$ such
that $|E(\Q)_{\tors}| \in R^{\times}$ and $\theta_S \in R[G_S]$.
For simplicity, we assume that $S$ is relatively prime to $N$. 
If $\theta_S \in I_S^{r_E}$, then we denote by $\tilde{\theta}_S$ the image of $\frac{1}{2}\theta_S$ in $I_{S^{+}}^{r_E}/I_{S^{+}}^{r_E+1},$
where $I_{S^{+}}$ denotes the augmentation ideal of $R[\Gal(\Q(\mu_S)^{+}/\Q)]$.
We note that our $\tilde{\theta}_S$ coincides with the leading coefficient considered in \cite{m-t} (cf.\ Remark \ref{different notation}).
For each positive divisor $T$ of $S$, 
we denote by $J_T$ the order of the cokernel of the natural map
\begin{equation*}
 E(\Q) \to \left(\oplus_{\ell|T}E(\mathbb{F}_{\ell})\right) \oplus \left(\oplus_{\ell|N}E(\Q_{\ell})/E_0(\Q_{\ell})\right).
\end{equation*}
\begin{conj}[Mazur-Tate]\label{full refined}Let $S$ be a square-free positive integer relatively prime to $N$.
Then $\theta_S \in I_S^{r_E}$, the Tate-Shafarevich group $\Sha$ of $E$ over $\Q$ is finite and
\begin{equation*}
\tilde{\theta}_S = |\Sha |\cdot \sum_{T|S>0}(-1)^{\nu(T)}J_T\cdot\eta_{r_E}\left( \mu_{S,T}(d_T)\right) \in I_{S^{+}}^{r_E}/I_{S^{+}}^{r_E+1},
\end{equation*}where $\nu(T)$ denotes the number of primes dividing $T$. 
See \cite[(2.5), (2.6), (3.1)]{m-t} for $d_T, \mu_{S,T}$ and $\eta_{r_E}$.
\end{conj}
\begin{rem}\label{remark on the conj}
\begin{enumerate}\item It may happen that $\theta_S \in I_S^{r_E+1}.$
We give some cases where it happens.
\begin{enumerate}
\item It is known that if  $|G_S| \in R^{\times}$ then  $I_S=I_S^2=I_S^3=\cdots$.
\item  Let  $\ell$ be a prime with $a_{\ell}=2$. Even if $r_E=0$, 
  Proposition \ref{mazur-tate} (1) implies that $\theta_{\ell} \in I_{\ell}$.
For example, if $E$ is defined by the equation
$y^2 + y = x^3 - x^2 - 2x + 1$, then 
$r_E=0,$ 
and the primes $ \ell \le100000$ satisfying $a_{\ell}=2$
are $\ell=\ $2, 3, 5, 251, 983, 1009, 1051, 1669, 8219, 9397, 10477, 11789, 14461,
21773, 24019, 32117, 51239, 57737, 93199, 95747, 97859, 98711. 
The calculation is due to Sage \cite{sage}.
\end{enumerate}
 It may also happen that the element $\eta_{r_E}(\mu_{S,T}(d_T))  \in I_{S^{+}}^{r_E}/I_{S^{+}}^{r_E+1}$ is trivial. 
Bertolini-Darmon \cite{b-d} constructed  a certain lift of $\eta_{r_E}(\mu_{S,T}(d_T))$ to $I_{S^{+}}^{r_E}$, which gives extra information in this case.
\item See \cite[Conjecture 4]{m-t} for more general cases.
Although Conjecture \ref{full refined} might look different from \cite[Conjecture 4]{m-t}, it is not difficult to check that they are equivalent.
\end{enumerate}
\end{rem}


\section{Darmon-Kolyvagin derivatives and Euler systems}\label{section derivatives}
In this section, we fix notation on derivatives and Euler systems, and  recall their properties.

We fix a prime $p\ge 5.$
For an integer $S$, we denote by $\Q(S)$ the maximal $p$-extension of $\Q$ inside $\Q(\zeta_{S})$ and put $\Gamma_{S}=\Gal(\Q(S)/\Q)$. 
For relatively prime integers $m$ and $n$, by the canonical decomposition $\Gamma_{mn}=\Gamma_m\times \Gamma_n$, we regard $\Gamma_m$ and  $\Gamma_n$ as subgroups of $\Gamma_{mn}$


\subsection{Darmon-Kolyvagin derivatives}\label{derivative}
Following \cite{dar92}, we introduce derivatives which we call Darmon-Kolyvagin derivatives as in \cite{l-v}.

As usual, for integers $j\ge 0$ and $k\ge 1$, we put
\begin{equation*}
\binom{j}{k}=\frac{j(j-1)\cdots (j-k+1)}{k !}.
\end{equation*}We put $\binom{j}{0}=1$ for $j\ge 0$.
For an element $\sigma \in \Gamma_{S}$ of order $n$ and for an integer $k\ge0$, we define
\begin{equation*}
D_{\sigma}^{(k)}=\sum_{j=0}^{n-1}\binom{j}{k}\sigma^{j}\in \Z[\Gamma_{S}].
\end{equation*}
We note that $D_{\sigma}^{(k)}=0$ if $k\ge n.$ 
 For $k<0$, we define $D_{\sigma}^{k}=0$.  
\begin{lem}\label{crucial}If $\sigma \in\Gamma_{S}$ is of order $n$ and $1\le k \le n-1$, then
\begin{equation*}
(\sigma-1)D_{\sigma}^{(k)}=\binom{n}{k}-\sigma D_{\sigma}^{(k-1)}.
\end{equation*}In particular, if $n$ is a power $q$ of $p$ and $0<k<p$,
then we have
\begin{equation*}
(\sigma-1)D_{\sigma}^{(k)}\equiv -\sigma D_{\sigma}^{(k-1)} \mod q.
\end{equation*}
\end{lem}
\begin{proof}This is proved by a straightforward computation.
For the second assertion, note that $\binom{q}{k} \equiv 0 \bmod q$ when $0<k<p$.
\end{proof}

\begin{dfn}\label{def of cond} 
In the following, we fix a generator $\sigma_{\ell}$ of $\Gamma_{\ell}$ for each prime $\ell$,
 and write $D_{\ell}^{(k)}=D_{\sigma_{\ell}}^{(k)}.$ 
Let $S>0$ be a square-free integer. 
We call an element  $D$ of $\Z[\Gamma_S]$ a \textit{Darmon-Kolyvagin derivative}, or simply, a derivative if $D$ is of the following form:
\begin{equation*}
D_{\ell_1}^{(k_1)}\cdots D_{\ell_s}^{(k_s)} \in \Z[\Gamma_{\ell_1\cdots\ell_s}]\subset \Z[\Gamma_S],
\end{equation*}where $\ell_1,\ldots, \ell_s$ are distinct primes dividing $S$, and each $k_i$ is an integer such that $0\le k_i < |\Gamma_{\ell_i}|$. 
We note that $\ell_1,\ldots, \ell_s, k_1,\ldots,k_s$ are uniquely determined.
We define 
\begin{equation*}
\mathrm{Supp}(D)=\ell_1\cdots\ell_s, \ \ \ \ \ \ \ \ \mathrm{Cond}(D)= \prod_{k_i>0}\ell_{i},
\end{equation*}
 which we call the \textit{support} and the \textit{conductor} of $D$, respectively.
We put
\begin{equation*}
\mathrm{ord}(D)  = k_1+\cdots+k_s, \ \  \  n(D) = \min_{k_i>0}\{|\Gamma_{\ell_i}|\},\ \ \ e_{\ell_i}(D)=k_i.
\end{equation*}
We call $\mathrm{ord}(D)$ the \textit{order} of $D$.
Since $\Gamma_{\ell}$ is a $p$-group for each prime $\ell$, the natural number $n(D)$ is a power of $p$. 
When $k_i=0$ for all $i$, we define $n(D)=1$.
When $S=\ell_1\cdots \ell_s$,  we define the norm operator as
 \begin{equation*}
 N_{S}=D_{\ell_1}^{(0)}\cdots D_{\ell_s}^{(0)}.
 \end{equation*} 
\end{dfn}

Let $S$ be a square-free positive integer and
$M$ a $\Zp[\Gamma_S]$-module without $p$-torsion. 
We take an element $a\in M,$ and put
\begin{equation*}
\theta = \sum_{\gamma \in \Gamma_S}\gamma a\otimes \gamma \in M\otimes_{\Zp} \Zp[\Gamma_S] .
\end{equation*}
The element $\theta$ has  a \textit{Taylor expansion} as follows.

\begin{prop}\label{taylor}Let $S=\ell_1\cdots \ell_s$ be the prime factorization of $S$.
Then, we have
\begin{equation*}
\theta = \sum_{\underline{k}=(k_1,\ldots,k_s)\in\Z_{\ge 0}^{\oplus s}} D_{\underline{k}}a\otimes (\sigma_{\ell_1}-1)^{k_1}\cdots (\sigma_{\ell_s}-1)^{k_s},
\end{equation*}where $D_{\underline{k}}:=D_{\ell_1}^{(k_1)}\cdots D_{\ell_s}^{(k_s)}$ for $\underline{k}=(k_1,\ldots,k_s).$
\end{prop}
\begin{rem}We note that $D_{\underline{k}}=0$ for all but finitely many $\underline{k} \in \Z_{\ge 0}^{\oplus s}$ 
\end{rem}

\begin{proof}We prove the proposition by induction on the number of primes dividing $S$,

We first assume that $S$ is a prime $\ell$ and put $\sigma=\sigma_{\ell}$.
Since $\Gamma_{\ell}$ is generated by $\sigma$,  we have
\begin{equation*}
\theta=\sum_{j=0}^{|\Gamma_{\ell}|-1}\sigma^j a\otimes \sigma^{j}.
\end{equation*}For each $j$, we note that 
\begin{equation*}
\sigma^j=(\sigma-1+1)^j=\sum_{k=0}^{j}\binom{j}{k}(\sigma-1)^k = \sum_{k\ge 0}\binom{j}{k}(\sigma-1)^k.
\end{equation*}
Hence, we have
\begin{align*}
\sum_{j=0}^{|\Gamma_{\ell}|-1}\sigma^j a\otimes \sigma^{j}&=\sum_{j=0}^{|\Gamma_{\ell}|-1}\sigma^j a\otimes \sum_{k\ge 0}\binom{j}{k}(\sigma-1)^k=\sum_{j=0}^{|\Gamma_{\ell}|-1}\sum_{k\ge 0}\binom{j}{k}\sigma^j a\otimes (\sigma-1)^k\\
&=\sum_{k\ge 0}\sum_{j=0}^{|\Gamma_{\ell}|-1}\binom{j}{k}\sigma^j a\otimes (\sigma-1)^k=\sum_{k\ge 0}D_{\ell}^{(k)} a\otimes (\sigma-1)^k.
\end{align*}Then, we complete the case where $S$ is a prime.

In the general case, we put $T=S/\ell_1.$
Then, we have
\begin{equation*}
\theta=
\sum_{\gamma_1\in \Gamma_{\ell_1}} \sum_{\gamma\in \Gamma_{T}}\gamma_1\gamma a \otimes \gamma \gamma_1.
\end{equation*}By the induction hypothesis,
\begin{equation*}
\sum_{\gamma\in \Gamma_{T}}\gamma a \otimes \gamma =\sum_{\underline{k}^{\prime}=(k_2,\ldots,k_s) \in \Z^{\oplus s-1}_{\ge 0}}D_{\underline{k}^{\prime}}a \otimes (\sigma_{\ell_2}-1)^{k_2}\cdots(\sigma_{\ell_s}-1)^{k_s}.
\end{equation*}
Hence, we have
\begin{align*}
\theta
&=\sum_{\gamma_1\in \Gamma_{\ell_1}}  \sum_{\gamma\in \Gamma_{T}}\gamma_1 \gamma a \otimes \gamma \gamma_1\\
&=\sum_{\gamma_1\in \Gamma_{\ell_1}} \ \sum_{\underline{k}^{\prime}=(k_2,\ldots,k_s) \in \Z^{\oplus s-1}_{\ge 0}}\gamma_1 D_{\underline{k}^{\prime}}a \otimes (\sigma_{\ell_2}-1)^{k_2}\cdots(\sigma_{\ell_s}-1)^{k_s} \gamma_1\\
&=  \sum_{\underline{k}^{\prime}} D_{\underline{k}^{\prime}} \sum_{\gamma_1\in \Gamma_{\ell_1}}  \gamma_1 a \otimes \gamma_1  (\sigma_{\ell_2}-1)^{k_2}\cdots(\sigma_{\ell_s}-1)^{k_s}\\
&\overset{(*)}{=}\sum_{\underline{k}^{\prime}} D_{\underline{k}^{\prime}} \sum_{k_1 \ge 0}  D_{\ell_1}^{(k_1)} a \otimes (\sigma_{\ell_1}-1)^{k_1}(\sigma_{\ell_2}-1)^{k_2}\cdots(\sigma_{\ell_s}-1)^{k_s}\\
&=\sum_{\underline{k}=(k_1,\ldots,k_s)\in\Z_{\ge 0}^{\oplus s}} D_{\underline{k}}a\otimes (\sigma_{\ell_1}-1)^{k_1}\cdots (\sigma_{\ell_s}-1)^{k_s},
\end{align*}where the equality $(*)$ also follows from the induction hypothesis.
\end{proof}

\begin{lem}\label{p-power of augmentation ideal}Let $G$ be a finite abelian $p$-group and $\sigma$ an element of $G$ with order $q$. Then,
\begin{equation*}
q(\sigma-1) \in I_{G}^p,
\end{equation*}where $I_G$ denotes the augmentation ideal of $\Zp[G].$
\end{lem}
\begin{proof}This is \cite[Lemma 3.5]{dar92}.
\end{proof}

Combining  Proposition \ref{taylor} and Lemma \ref{p-power of augmentation ideal}, we have the following.
\begin{lem}\label{fromcongtovanishing}Let $t\ge1$. Assume that for all Darmon-Kolyvagin derivatives $D$ such that $\Supp(D)=S$ and  $\mathrm{ord}(D)<\min\{t,p\},$ we have $Da \equiv 0  \mod n(D)$. 
Then,
\begin{equation*}
\theta -N_S a \otimes 1 \in M\otimes_{\Zp}I_{\Gamma_S}^{\min\{t,p\}}.
\end{equation*}
\end{lem}
\begin{rem}This is \cite[Lemma 3.8]{dar92}. 
It seems that there is an error in the statement of \cite[Lemma 3.8]{dar92}.
However, the error is not crucial when we consider  Euler systems.
\end{rem}
\begin{proof}As in Proposition \ref{taylor}, we write
\begin{equation}\label{taylor again}
\theta = \sum_{\underline{k}=(k_1,\ldots,k_s)\in\Z_{\ge 0}^{\oplus s}} D_{\underline{k}}a\otimes (\sigma_{\ell_1}-1)^{k_1}\cdots (\sigma_{\ell_s}-1)^{k_s}.
\end{equation}
We pick an element $\underline{k}=(k_1,\ldots,k_s) \in \Z_{\ge0}^{\oplus s}\setminus \{0,\ldots,0\}$ such that $k_1+\cdots + k_s < \min\{t,p\}$,
that is, $0<\mathrm{ord}(D_{\underline{k}}) < \min\{t,p\}.$
By the definition of $n(D_{\underline{k}}),$ there exists $i$ such that  $|\Gamma_{\ell_i}|=n(D_{\underline{k}})$ and $k_i >0$.
Since  $D_{\underline{k}}a \equiv 0 \bmod n(D_{\underline{k}})$, Lemma \ref{p-power of augmentation ideal} implies that
\begin{equation}\label{vanishing of da}
D_{\underline{k}}a\otimes  (\sigma_{\ell_1}-1)^{k_1}\cdots (\sigma_{\ell_s}-1)^{k_s} \in M\otimes I_{\Gamma_S}^{p}.
\end{equation}
This holds  for each $D_{\underline{k}}$ such that  $0 <\mathrm{ord}(D_{\underline{k}}) <\min\{t,p\}$.
By (\ref{taylor again}), we complete the proof.
\end{proof}




\subsection{Preliminaries on Galois cohomology}
We denote by $T$ the $p$-adic Tate module $T_p(E)$ of $E$.
In the rest of Section \ref{section derivatives}, we assume that the Galois representation 
$\rho:G_{\Q} \to \mathrm{Aut}_{\Zp}(T)$
is surjective.
Then by \cite[Proposition 3.5.8 (ii)]{rub}, 
there exists an element $\tau \in \Gal(\overline{\Q}/\Q(\mu_{p^{\infty}}))$ such that
\begin{equation}\label{tau}
T/(\tau-1)T\cong \Zp.
\end{equation}

  \begin{prop}\label{trivial torsion}For a power $q$ of $p$ and a finite abelian extension $F$ of $\Q,$ we have $E(F)[q]=0.$
 Moreover, the restriction map induces an isomorphism
 \begin{equation*}
 H^1(\Q,E[q]) \cong H^0\left(F/\Q, H^1(F,E[q])\right).
 \end{equation*}
 \end{prop}
 \begin{proof}For the first assertion, we only need to show that  $E(F)[p]=0.$
 We assume that $E(F)[p] \not=0$, and take a non-trivial point $P \in E(F)[p].$
Since the Galois representation $G_{\Q} \to \mathrm{Aut}_{\Z/p\Z}(E[p])$ is surjective, 
for each non-trivial point $Q \in E[p]$ there exists an element $\sigma \in G_{\Q}$ such that $\sigma P=Q$.
Since the extension $F/\Q$ is a Galois extension, we have $Q \in E(F)[p].$
Thus, $\Q(E[p]) \subseteq F$, which implies that $\Gal(\Q(E[q])/\Q)$ is abelian.
 However, since $\Gal (\Q(E[p])/\Q) \cong \mathrm{GL}_2(\Z/p\Z)$ is not abelian, we have a contradiction. Then, we show that  $E(F)[q]=0.$
The second assertion follows from the exact sequence  
\begin{equation*}
0\to H^1(F/\Q,E(F)[q]) \to H^1(\Q,E[q]) \to H^0\left(F/\Q, H^1(F,E[q])\right) \to H^2(F/\Q,E(F)[q])
\end{equation*}which is induced by the  Hochschild-Serre spectral sequence.
 \end{proof}

For a torsion module $M$ and an element  $b\in M$, we denote by $\ord(b,M)$ the order of $b$.
\begin{lem}\label{order}Let $q$ be a power of $p$ and $L$   a finite Galois extension of $\Q$ such that $G_L$ acts trivially on $E[q]$.
Then for $\ka , \eta \in H^1(\Q,E[q])$, there exists an element $\gamma$ of $G_L$ such that
\begin{enumerate}
\item $\ord\left(\ka(\gamma\tau),E[q]/(\tau-1)E[q]\right) \ge \ord(\ka, H^1(L,E[q])),$
\item $\ord\left(\eta(\gamma\tau),E[q]/(\tau-1)E[q]\right) \ge \ord(\eta, H^1(L,E[q])),$
\end{enumerate}where $\tau$ is as in $($\ref{tau}$)$, and we regard $\ka,\eta$ as elements of $H^1(L,E[q]))$ by the restriction map $H^1(\Q,E[q])\to H^1(L,E[q])$.
\end{lem}
\begin{rem}For $\gamma \in G_L$ and $\ka \in H^1(\Q,E[q]),$ the image of $\ka(\gamma\tau)$ in $E[q]/(\tau-1)E[q]$ is independent of the choice of a cocycle representing $\ka$.
\end{rem}
\begin{proof}This is \cite[Lemma 5.2.1]{rub}.
\end{proof}

\subsection{Euler systems}
For a prime $\ell$, we define $P_{\ell}(t) \in \Z[t]$ by
\begin{equation}\label{PLT}
P_{\ell}(t)= 1-a_{\ell}t+ \epsilon (\ell)t^2,
\end{equation}where $a_{\ell}$ and $\epsilon$ are as in Section \ref{chapter of refined BSD}.
Let $\Sigma$ be a finite set of primes which contains all the primes dividing $pN$.
We put
\begin{align*}
\mathscr{R} =\{\text{primes}\ \ell ;\ \ell \notin \Sigma\},\ \
\mathscr{N} =\{\text{square-free\ products\ of\ primes\ in}\ \mathscr{R}\}\cup\{1\}.
\end{align*}

\begin{dfn}\label{def of euler system}We call $\{z_{Sp^n}\}_{S\in \mathscr{N},n\ge 0} \in \prod_{S,n} H^1(\Q(Sp^n),T)$ an  \textit{Euler system} (for $T$ and $\mathscr{N}$)  if $\{z_{Sp^n}\}$ satisfies the following conditions.
\begin{enumerate}
\item For $S \in \mathscr{N}$, a prime $\ell\in\mathscr{R}$ not dividing $S$, and $n\ge0$, we have
\begin{equation*}
\mathrm{Cor}_{S\ell /S}z_{S\ell p^n}=P_{\ell}(\Fr_{\ell}^{-1})z_{Sp^n},
\end{equation*}where $\mathrm{Cor}_{S\ell p^n/Sp^n}: H^1(\Q(S\ell p^n),T) \to H^1(\Q(Sp^n),T)$ denotes the corestriction map, and $\Fr_{\ell} \in \Gamma_{Sp^n}$ denotes the arithmetic Frobenius at $\ell$.
\item For $S \in \mathscr{N}$, the system $\{z_{Sp^n}\}_{n\ge 0}$ is a norm compatible system, that is, 
\begin{equation*}
\{z_{Sp^n}\}_{n\ge0} \in \varprojlim_{n}H^1(\Q(Sp^n),T),
\end{equation*}where the limit is taken with respect to  the corestriction maps $\mathrm{Cor}_{Sp^{n+1}/Sp^n}$.
\end{enumerate}
\end{dfn}

\begin{rem}\label{another norm relation}
Our definition of Euler system slightly differs from the usual definition in \cite{kat99},  \cite{p-r98} and  \cite{rub}.
In the usual definition, instead of the condition (1) in Definition \ref{def of euler system}, every Euler system is required to satisfy 
\begin{equation*}
\mathrm{Cor}_{S\ell p^n/Sp^n}(z_{S\ell p^n})= \left(1-\frac{a_{\ell}}{\ell}\Fr_{\ell}^{-1}+\frac{1}{\ell}\Fr_{\ell}^{-2}\right)z_{Sp^n}.
\end{equation*}
However, since $P_{\ell}(t)\equiv \left(1-\frac{a_{\ell}}{\ell}t+\frac{1}{\ell}t^{2}\right) \mod \ell-1$, 
by \cite[Lemma 9.6.1]{rub}, 
the existence of an Euler system in our sense is equivalent to the existence of an Euler system in the usual sense.
 \end{rem}

For a local field $K$  and a topological module $M$ with a continuous $G_{K}$-action, we put
\begin{align*}
H^1_{\ur}(K,M)&=\mathrm{ker}\left(H^1(K,M) \to H^1(K^{\ur},M)\right).
\end{align*}

\begin{prop}\label{inverse limit}Let $\{z_{Sp^n}\}$ be an Euler system and  $\lambda \nmid p$ a prime  of $\overline{\Q}$.
Then,  for $S\in \mathscr{N}$ and $n\ge 0,$ the image $\loc_{\lambda}(z_{Sp^n})$ of $z_{Sp^n}$ in $H^1(\Q(Sp^n)_{\lambda},T)$ is unramified, that is, 
\begin{equation*}
\loc_{\lambda}(z_{Sp^{n}}) \in H^1_{\ur}(\Q(Sp^n)_{\lambda},T),
\end{equation*}where $\Q(Sp^n)_{\lambda}$ denotes the completion at the prime of $\Q(Sp^n)$ below $\lambda$.
\end{prop}
\begin{proof}This is \cite[Corollary B.3.5]{rub}.
\end{proof}

\subsection{Local behavior  at primes not dividing $p$}\label{local condition of derivative classes}
In this subsection, following \cite[Chapter 4]{rub}, we study local behavior of derivatives of Euler systems at primes \textit{not} dividing $p$.

Let $q$ be a power of $p$ and $\{z_{Sp^n}\}_{S \in \mathscr{N},n\ge 0}$ an Euler system.
For an integer $S \in \mathscr{N}$, an element $x \in H^1(\Q(S),E[q])$ and a prime $\lambda$ of $\Q(S)$, we denote by $\loc_{\lambda}(x)$ the image of $x$ in $H^1(\Q(S)_{\lambda},E[q]).$

\subsubsection{Unramifiedness of derivatives  at primes not dividing conductors}
For a prime $\ell$, we put $m_{\ell}=[E(\Q_{\ell}):E_0(\Q_{\ell})].$
We assume  that
\begin{equation*}
p\nmid 6N\prod_{\ell|N}m_{\ell}.
\end{equation*}
\begin{prop}\label{unramified}Let $D$ be a Darmon-Kolyvagin derivative with support $S$ and conductor
 $S^{\prime}$.
We assume that $Dz_S \bmod q \in H^0\left({\Gamma_S}, H^1(\Q(S),T)/q\right)$ and
 denote by $\ka \in H^1(\Q,E[q])$ the inverse image of $Dz_S \bmod q \in H^1(\Q(S),E[q])$ under the isomorphism $($cf.\ Proposition \ref{trivial torsion}$)$
$H^1(\Q,E[q]) \cong H^0\left(\Gamma_S,H^1(\Q(S),E[q])\right)$.
Then for every prime $\ell\nmid pS^{\prime}$,
\begin{equation*}
\loc_{\ell}(\ka) \in H^1_{\ur}(\Q_{\ell},E[q]).
\end{equation*}
\end{prop}
\begin{rem}By taking Galois cohomology with respect to the exact sequence
\begin{equation*} 
0\to T \xrightarrow{\times q} T \to E[q] \to 0,
\end{equation*}we have an inclusion $H^1(\Q(S),T)/q \hookrightarrow H^1(\Q(S),E[q]),$ by which we
regard $Dz_S \bmod q$ as an element of $H^1(\Q(S),E[q]).$
\end{rem}
\begin{proof}First, we suppose that $\ell \nmid pS$.
Since the extension $\Q(S)/\Q$ is unramified at $\ell$, we have $(\Q(S)_{\lambda})^{\ur} \cong\Q_{\ell}^{\ur}$ for a prime $\lambda|\ell.$ Hence, we have $\loc_{\ell}(\ka) =\loc_{\lambda}(Dz_S)$ as elements of $H^1(\Q_{\ell}^{\ur},E[q])$.
Then, by Proposition \ref{inverse limit},
we have $\loc_{\ell}(\ka) \in H^1_{\ur}(\Q_{\ell},E[q]).$

We next consider a prime $\ell$ dividing $S/S^{\prime}$.
Then we have 
\begin{equation*}
Dz_S=D^{\prime}N_{\ell}z_{S}=P_{\ell}(\Fr_{\ell}^{-1})D^{\prime}z_{S/\ell},
\end{equation*}where $D^{\prime}$ is a derivative such that $\mathrm{Supp}(D^{\prime})=S/\ell$.
Since the extension $\Q(S/\ell)/\Q$ is unramified at $\ell$,
 for a prime $\lambda$ of $\Q(S/\ell)$ we have $\loc_{\ell}(\ka)=\loc_{\lambda}(D^{\prime}P_{\ell}(\Fr_{\ell}^{-1})z_{S/\ell})$ as elements of $H^1(\Q_{\ell}^{\mathrm{ur}},E[q])$.
Then  by Proposition \ref{inverse limit}, we complete the proof.
\end{proof}

\begin{cor}\label{f}
Under the notation as above, for every prime $\ell \nmid pS^{\prime}$, we have
\begin{equation*}
\loc_{\ell}(\ka) \in E(\Q_{\ell})/q,
\end{equation*}where $E(\Q_{\ell})/q$ is regarded as a subgroup of $H^1(\Q_{\ell},E[q])$ by the Kummer map.
\end{cor}
\begin{proof}Our proof is based on that of \cite[Theorem 4.9]{dar92}.
By the exact sequence
\begin{equation*}
0 \to E(\Q_{\ell})/q\to H^1(\Q_{\ell},E[q]) \to H^1(\Q_{\ell},E)[q]\to 0, 
\end{equation*}
it suffices to show that the image $\loc_{/f,\ell}(\ka)$ of $\ka$ in $H^1(\Q_{\ell},E)[q]$ is trivial.
 Proposition \ref{unramified} shows that  $\loc_{/f,\ell}(\ka) \in H^1(\Q_{\ell}^{\mathrm{ur}}/\Q_{\ell},E(\Q_{\ell}^{\mathrm{ur}}))[q]$.
Since $p\nmid m_{\ell}$ (if $\ell \nmid N$, then $m_{\ell}=1$), by using \cite[Chapter I, Proposition 3.8]{Mil},
we have
 \begin{equation*}
H^1(\Q_{\ell}^{\mathrm{ur}}/\Q_{\ell},E(\Q_{\ell}^{\mathrm{ur}}))[q]= 0,
   \end{equation*}and hence $\loc_{/f,\ell}(\ka)=0.$
\end{proof}

\subsubsection{Local behavior at primes dividing conductors}
 We put
\begin{gather}\label{def of rqe}
\begin{split}
\mathscr{R}_q&=\{ \ell\in \mathscr{R} \ ; \ q|\ell-1 \},\ \ \
\mathscr{R}_{E,q}=\{ \ell\in \mathscr{R}_q \ ; \ q|P_{\ell}(1)\},\\
\mathscr{N}_q&=\{ \text{square-free products of primes in } \mathscr{R}_q \}.
\end{split}
\end{gather}
We take an integer $S \in \mathscr{N}_p$.

\begin{dfn}\label{def of universal euler system}
For a positive divisor $S^{\prime}$ of $S$, let $x_{S^{\prime}}$ denote an indeterminate. 
We denote by $Y_{S}$ the free $\Zp[\Gamma_{S}]$-module generated by $\{x_{S^{\prime}}\}_{S^{\prime}|S>0},$
that is, $Y_{S}=\oplus_{S^{\prime}|S>0}\Zp[\Gamma_{S}]x_{S^{\prime}}.$
We denote by $Z_{S}$ the $\Zp[\Gamma_S]$-submodule of  $Y_{S}$ generated by the following elements:
\begin{equation*}
 \sigma x_{S^{\prime}}-x_{S^{\prime}}\ \text{for}\ S^{\prime}|S\ \text{and}\ \sigma\in \Gamma_{S/S^{\prime}}, \  \ \ \ N_{\ell}x_{ S^{\prime}\ell }-P_{\ell}(\mathrm{Fr}_{\ell}^{-1})x_{S^{\prime}}\ \text{for primes}\ \ell \ \text{with}\  \ell S^{\prime}|S.
\end{equation*} 
 We define  $X_{S}=Y_{S}/Z_{S}.$
\end{dfn}

If we regard $z_{S^{\prime}}$  as an element of $H^1(\Q(S),T)$ for $S^{\prime}|S$ by the restriction map,
 then there exists a unique homomorphism of $\Gamma_{S}$-modules
\begin{equation*}
g_{S}:X_{S} \to H^1(\Q(S),T)
\end{equation*}sending $x_{S^{\prime}}$ to $z_{S^{\prime}}$ for $S^{\prime}|S$.

Let $q$ be a power of $p$, and we put  $M_q =\mathrm{Ind}^{G_{\Q}}_{\{1\}}(E[q])$.
We recall that the $G_{\Q}$-module $\mathrm{Ind}^{G_{\Q}}_{\{1\}}(E[q])$ is defined as the module of continuous maps from $G_{\Q}$ to $E[q]$, and $G_{\Q}$ acts on $\mathrm{Ind}^{G_{\Q}}_{\{1\}}(E[q])$ by $(\sigma f)(g) =f(g\sigma )$ for $\sigma,g \in G_{\Q}.$ 
Then we have an exact sequence of $G_{\Q}$-modules
\begin{equation*}
0\to E[q] \to M_q \to M_q/E[q] \to 0,
\end{equation*}where the map $E[q] \to M_q$ is defined as $y\mapsto (g\mapsto gy ).$ 
For a finite extension $L$ of $\Q$, 
by taking Galois cohomology, 
we obtain an exact sequence
\begin{equation}\label{connecting map}
0 \to E(L)[q] \to M_q^{G_{L}} \to (M_q/E[q])^{G_{L}} \xrightarrow{\delta_L} H^1(L,E[q])\to 0.
\end{equation}See \cite[Proposition B.4.5]{rub} for the surjectivity of the connecting map $\delta_L$.

\begin{prop}\label{lift}There exists  a $\Gamma_{S}$-homomorphism $d_{S}$  from $X_{S}$ to $(M_q/E[q])^{G_{\Q(S)}}$ making the following diagram  commutative$:$
\[
\xymatrix{
& & (M_q/E[q])^{G_{\Q(S)}} \ar[d]^{\delta_{\Q(S)}}  \\
  X_{S}  \ar[rru]^{d_{S}} \ar[rr]_{g_{S}} & &   H^1(\Q(S),E[q]). \\
}
\]
\end{prop}
\begin{proof}This is \cite[Proposition 4.4.8]{rub}.
\end{proof}

We take a prime $\ell\in \mathscr{R}_{E,q}$ which splits completely in $\Q(S)$.
We denote by $\mathscr{D}_{\ell} \subseteq G_{\Q}$ a decomposition group  of $\ell$, and by $\mathscr{I}_{\ell}\subset \mathscr{D}_{\ell}$ the inertia group.
Then, the natural map $\mathscr{I}_{\ell} \to \Gamma_{\ell}$ is surjective.
We fix a lift of $\sigma_{\ell}$ to $\mathscr{I}_{\ell}$, which we also denote by $\sigma_{\ell}.$ 
We fix a lift $\Fr_{\ell} \in \mathscr{D}_{\ell}$ of the arithmetic Frobenius at $\ell$.
By abuse of notation, we put
\begin{equation*}
N_{\ell}=\sum_{i=1}^{n_{\ell}}\sigma_{\ell}^i,\ \ \ D_{\ell}^{(1)}=\sum_{i=0}^{n_{\ell}-1}i\sigma_{\ell}^i \ \ \in \Z[\mathscr{I}_{\ell}],
\end{equation*}where $n_{\ell}:=|\Gamma_{\ell}|$.
Then we have
\begin{equation}\label{derivative2}
(\sigma_{\ell}-1)D_{\ell}^{(1)}=n_{\ell}\sigma_{\ell}^{n_{\ell}}-N_{\ell}\ \ \ \text{in}\ \Z[\mathscr{I}_{\ell}].
\end{equation}

Let $H^1_f(\Q_{\ell},E[q])$ denote the image of the Kummer map $E(\Q_{\ell})/q \to H^1(\Q_{\ell},E[q])$. 
We put
\begin{equation*}
H^1_{/f}(\Q_{\ell},E[q])=H^1(\Q_{\ell},E[q])/H^1_f(\Q_{\ell},E[q]),
\end{equation*}which is isomorphic to $H^1(\Q_{\ell},E)[q].$
It is known that 
$H^1_{f}(\Q_{\ell},E[q])=H^1_{\ur}(\Q_{\ell},E[q])$, 
and hence $H^1_{/f}(\Q_{\ell},E[q])=H^1(\Q_{\ell}^{\ur},E[q]).$
By \cite[Lemma 1.4.7 (i)]{rub}, we have two isomorphisms 
\begin{align*}
& \alpha_{\ell}:H^1_{/f}(\Q_{\ell},E[q]) \cong E[q]^{\mathrm{Fr}_{\ell}=1};\ \ c \mapsto c(\sigma_{\ell}), \\
& \beta_{\ell}:H^1_f(\Q_{\ell},E[q]) \cong E[q]/(\mathrm{Fr}_{\ell}-1)E[q];\ \ c \mapsto c(\mathrm{Fr}_{\ell}),
\end{align*}where each element $c \in H^1(\Q_{\ell},E[q])$ is regarded as a cocycle.
Here, we note that the map $\alpha_{\ell}$ depends on the choice of $\sigma_{\ell}$.
Since $\ell \in \mathscr{R}_{E,q}$, we have $P_{\ell}(1) = 2-a_{\ell}\equiv 0 \mod q,$ and then $a_{\ell } \equiv 2 \mod q.$
Hence, $P_{\ell}(t) \equiv (t-1)^{2} \bmod q$.
Since $P_{\ell}(t) \equiv \mathrm{det}_{\Zp}(1-\Fr_{\ell}t|T) \mod q$, we have
$P_{\ell}(\mathrm{Fr}_{\ell}^{-1})E[q]=0.$
Therefore, if we put $Q_{\ell}(t) = t-1$, then we have a homomorphism
\begin{equation*}
Q_{\ell}(\mathrm{Fr}_{\ell}^{-1}): E[q]/(\mathrm{Fr}_{\ell}-1)E[q] \to E[q]^{\mathrm{Fr}_{\ell}=1}.
\end{equation*}
We define a homomorphism
\begin{equation*}
\phi^{fs}_{\ell}: H^1_f(\Q_{\ell},E[q]) \to H^1_{/f}(\Q_{\ell},E[q])
\end{equation*} as the composite
\begin{equation*}
H^1_f(\Q_{\ell},E[q])\xrightarrow{\beta_{\ell}}E[q]/(\mathrm{Fr}_{\ell}-1)E[q]\xrightarrow{Q_{\ell}(\mathrm{Fr}_{\ell}^{-1})}E[q]^{\mathrm{Fr}_{\ell}=1} \xrightarrow{\alpha_{\ell}^{-1}} H^1_{/f}(\Q_{\ell},E[q]).
\end{equation*}

For each Darmon-Kolyvagin derivative $D$, we fix a lift $D$ to  $ \Z[G_{\Q}]$.

\begin{thm}\label{ell}Let $S$ be an element of $\mathscr{N}_p$ and $q$ a power of $p$. 
We take a prime  $\ell \in \mathscr{R}_{E,q}$ which splits completely in $\Q(S).$
Let $\lambda$ be the prime of $\Q(S)$ above $\ell$ corresponding to the decomposition group $\mathscr{D}_{\ell}$ of $\Q.$ 
For a Darmon-Kolyvagin derivative $D$ whose support is $S$, we have the following.
\begin{enumerate}
\item $\loc_{\lambda}(Dz_S \bmod q) \in H^1_f(\Q(S)_{\lambda},E[q]) = H^1_f(\Q_{\ell},E[q]).$
\item $DD_{\ell}^{(1)}z_{S\ell} \bmod q\in H^0\left(\Gamma_{\ell}, H^1(\Q(S\ell),T)/q\right).$
\item If $\ka^{(\ell)} \in H^1(\Q(S),E[q])$  denotes the inverse image of $DD_{\ell}^{(1)} z_{S\ell} \bmod q$ under the isomorphism $H^1(\Q(S),E[q]) \cong H^0\left(\Gamma_{\ell}, H^1(\Q(S\ell),E[q])\right)$ and $\loc_{/f,\lambda}(\ka^{(\ell)})$ denotes the image of $\ka^{(\ell)}$ in $H^1_{/f}(\Q_{\ell},E[q]),$ then we have
\begin{equation*}
\loc_{/f,\lambda}(\ka^{(\ell)})=\phi^{fs}_{\ell}(\loc_{\lambda}(Dz_S \bmod q)).
\end{equation*}
\item In addition, if $E[q]/(\Fr_\ell-1)E[q]\cong \Z/q\Z,$
 then 
\begin{equation*}
\ord(\loc_{\lambda}(\ka^{(\ell)}),H^1(\Q_{\ell},E)[q])=\ord(\loc_{\lambda}(Dz_S),H^1(\Q_{\ell},E[q])).
\end{equation*}
\end{enumerate} 
\end{thm}

\begin{proof}
We follow the proof of \cite[Theorem 4.5.4]{rub}.
The assertion (1) follows from Proposition \ref{inverse limit}.
By Lemma \ref{crucial}, we have
\begin{equation*}
(\sigma_{\ell}-1)DD_{\ell}^{(1)}x_{S\ell} \equiv -\sigma_{\ell} DN_{\ell}x_{S\ell} \equiv-\sigma_{\ell}DP_{\ell}(\Fr_{\ell}^{-1})x_{S} \equiv -\sigma_{\ell} DP_{\ell}(1)x_{S} \equiv 0 \mod qX_{S\ell},
 \end{equation*} where the third congruence follows from $\mathrm{Fr}_{\ell}=1$ in $ \Gamma_S$, and the last congruence follows from $\ell\in \mathscr{R}_{E,q}$.
Hence, $DD_{\ell}^{(1)}x_{S\ell} \bmod q \in \left(X_{S\ell}/q\right)^{\Gamma_{\ell}}$.
Since a homomorphism  $d_{S\ell}$ as in Proposition  \ref{lift} is $G_{\Q}$-equivariant, we have $d_{S\ell}(DD_{\ell}^{(1)}x_{S\ell}) \in H^0(\Gamma_{\ell}, (M_q/E[q])^{G_{\Q(S\ell)}}).$
Since $\delta_{\Q(S\ell)}(d_{S\ell}(x_{S\ell}))=z_{S\ell} \bmod q,$
we have
\begin{equation*}
DD_{\ell}^{(1)}z_{S\ell} \bmod q  \in H^0\left(\Gamma_{\ell}, H^1(\Q(S\ell),T)/q\right).
\end{equation*}

We take lifts $\hat{d}(x_{S\ell}),\hat{d}(x_{S}) \in M_q$ of $d_{S\ell}(x_{S\ell}), d_{S\ell}(x_{S})$, respectively.
By the definition of $\phi^{fs}_{\ell}$,  it suffices to show that
\begin{equation}\label{equality of lifts}
Q_{\ell}(\mathrm{Fr}_{\ell}^{-1})\left((Dz_S \bmod q)(\mathrm{Fr}_{\ell}\right) )=\ka^{(\ell)}(\sigma_{\ell})\ \in E[q],
\end{equation}where we regard $Dz_S$ and $\ka$ as cocycles.
Since $\delta_{\Q(S)}$ is the connecting map from $(M_q/E[q])^{G_{\Q(S)}}$ to $H^1(\Q(S),E[q])$, we have
\begin{equation}\label{dzs}
(Dz_S \bmod q)(\mathrm{Fr}_{\ell}) =(\mathrm{Fr}_{\ell}-1)D\hat{d}(x_{S}), \ \ 
\ka^{(\ell)}(\sigma_{\ell})=(\sigma_{\ell}-1)D_{\ell}^{(1)}D \hat{d}(x_{S\ell}) \in E[q].
\end{equation}
Since $Q_{\ell}(\mathrm{Fr}_{\ell}^{-1})(\Fr_{\ell}^{-1}-1)E[q]=P_{\ell}(\mathrm{Fr}_{\ell}^{-1})E[q]=0,$ we have 
\begin{equation*}
Q_{\ell}(\mathrm{Fr}_{\ell}^{-1})\left((Dz_S \bmod q)(\mathrm{Fr}_{\ell}\right) )=Q_{\ell}(\mathrm{Fr}_{\ell}^{-1})\Fr_{\ell}^{-1}\left((Dz_S \bmod q)(\mathrm{Fr}_{\ell}\right)).
\end{equation*}
 Thus, by (\ref{dzs}), (\ref{derivative2}) and \cite[Lemma 4.7.1]{rub}, we obtain
\begin{align*}
&\ \ \ \ Q_{\ell}(\mathrm{Fr}_{\ell}^{-1})((Dz_S \bmod q)(\mathrm{Fr}_{\ell}))-\ka^{(\ell)}(\sigma_{\ell})\\
&=Q_{\ell}(\mathrm{Fr}_{\ell}^{-1})\mathrm{Fr}_{\ell}^{-1}((Dz_S \bmod q)(\mathrm{Fr}_{\ell}))-\ka^{(\ell)}(\sigma_{\ell})\\
&=Q_{\ell}(\mathrm{Fr}_{\ell}^{-1})\mathrm{Fr}_{\ell}^{-1}(\mathrm{Fr}_{\ell}-1)D\hat{d}(x_{S})-(\sigma_{\ell}-1)D_{\ell}^{(1)}D\hat{d}(x_{S\ell})\\
&=-P_{\ell}(\mathrm{Fr}_{\ell}^{-1})D\hat{d}(x_{S})+N_{\ell}D\hat{d}(x_{S\ell}).
\end{align*}
By \cite[Lemma 4.7.3]{rub}, this is zero, and then we conclude (\ref{equality of lifts}).

If $E[q]/(\Fr_{\ell}-1)E[q] \cong \Z/q\Z$, then \cite[Corollary A.2.7]{rub} says that $\phi_{\ell}^{fs}$ is an isomorphism, and hence the assertion (4) follows.
\end{proof}

\section{Divisibility of Euler systems for elliptic curves}\label{main chapter}
In this section, 
we show that certain derivatives of Euler systems are divisible by a power of $p$ (Theorem \ref{keythm}), and give applications.

We keep the notation as in Section \ref{section derivatives}.
In particular, let $\{z_{Sp^n}\}_{S\in \mathscr{N},n\ge 0}$ be an Euler system for $T$ and some $\mathscr{N}$ in the sense of Definition \ref{def of euler system}.

\subsection{The theorem on divisibility of  Euler systems}\label{main section}
The aim of this subsection is to prove Theorem \ref{keythm}. 
We also give a modification (Theorem \ref{general}) of Theorem \ref{keythm}, which is used to prove Theorems \ref{sel} and \ref{theorem on leading term}.

\subsubsection{Notation}
\begin{assume}\label{main assumption}
Throughout Section \ref{main chapter}, we assume that $p\nmid 6N \prod_{\ell |N}m_{\ell}$ and that 
 the Galois representation $G_{\Q} \to \mathrm{Aut}_{\Zp}(T)$  is surjective.
 \end{assume}

Let $q$ be a power of $p$.
\begin{dfn}\label{rq}
For a finitely generated $\Zp$-module $M$, we define an integer $r_q(M)$ by
\begin{equation*}
M\otimes \Z/q\Z \cong (\Z/q\Z)^{\oplus r_q(M)} \oplus M^{\prime},
\end{equation*}where the exponent of $M^{\prime}$ is strictly less than $q$.
\end{dfn}
\begin{lem}\label{r}For an exact sequence of finite $\Z/q\Z$-modules
$0\to M^{\prime} \to M \to M^{\prime\prime}$, we have
\begin{equation*}
r_q(M)\le r_q(M^{\prime})+r_p(M^{\prime\prime}). 
\end{equation*}
\end{lem}
\begin{proof}This is \cite[Lemma 5.1]{dar92}.
\end{proof}

\begin{dfn}We define the Selmer group $\Sel(\Q,E[q])$ by
\begin{equation*}
\Sel(\Q,E[q]) = \mathrm{ker}\left(H^1(\Q,E[q]) \to \prod_{w: \text{places}} \frac{H^1(\Q_{w},E[q])}{E(\Q_{w})/q}\right),
\end{equation*}
and for a positive integer $S$, we define a subgroup $H^1_{f,S}(\Q,E[q])$ of $\Sel(\Q,E[q])$ by 
\begin{equation}
H^1_{f,S}(\Q,E[q])=\mathrm{ker}\left(\Sel(\Q,E[q]) \to \oplus_{\ell|S}E(\Q_{\ell})/q\right),
\end{equation}where $\ell$ ranges over all the primes dividing $S$.
If there is no fear of confusion, we simply write $H^1_{f,S}=H^1_{f,S}(\Q,E[q]).$
\end{dfn}

We put 
$A_q(S)=\oplus_{\ell|S}E(\Q_{\ell})/q.$
\begin{lem}\label{variation of r}Let $S$ be a positive integer and $\ell\nmid S$ a prime such that $E(\Q_\ell)/p$ is cyclic $($i.e. the module $E(\Q_{\ell})/p$ is trivial or isomorphic to $\Z/p\Z$$)$.
Then, we have
\begin{equation}\label{minus 1}
r_q(H^1_{f,S\ell})+r_p(A_q(S\ell))-1\le r_q(H^1_{f,S})+r_p(A_q(S)).
\end{equation}
In addition, if $E(\Q_\ell)/p\cong \Z/p$, then we have
\begin{equation}\label{usually for l}
r_q(H^1_{f,S})+r_p(A_q(S))\le r_q(H^1_{f,S\ell})+r_p(A_q(S\ell))
\end{equation}
\end{lem}
\begin{proof}Since $H^1_{f,S\ell} \subseteq H^1_{f,S}$ and  $r_p(A_q(S\ell))\le r_p(A_q(S))+1$, we have (\ref{minus 1}).
We assume that $E(\Q_\ell)/p\cong \Z/p$.
By Lemma \ref{r} with the exact sequence  $0\to H^1_{f,S\ell} \to H^1_{f,S} \to E(\Q_{\ell})/q$,
we have 
\begin{center}
$ r_q(H^1_{f,S}) \le r_q(H^1_{f,S\ell})+r_p(E(\Q_{\ell})/p)=r_q(H^1_{f,S\ell})+1.$ 
\end{center}Since $r_p(A_q(S))+1=r_p(A_q(S\ell)),$
we deduce that
\begin{equation*}
r_q(H^1_{f,S}) +r_p(A_q(S)) \le r_q(H^1_{f,S\ell})+r_p(A_q(S\ell)).
\end{equation*}
\end{proof}

\begin{dfn}Let $S \in \mathscr{N}_q$ (see (\ref{def of rqe}) for the notation). For a Darmon-Kolyvagin derivative $D$ whose support is $S$, we define the \textit{weight} of $D$ as
\begin{equation*}
w(D)=\mathrm{ord}(D)-|\{\ell \in \mathscr{R}_{E,q} ; \ \ell \ \text{divides}\ S \}|.
\end{equation*}
\end{dfn}
\begin{rem}In Darmon's argument, the notion of weight also played an important role. 
We modify his weight for our case.
\end{rem}

\begin{prop}\label{negative}
Let $D$ be a Darmon-Kolyvagin derivative with support $S$. 
Suppose that $S\in\mathscr{N}_q$. 
 If $w(D)<0$ and $\max_{\ell|S} \{e_{\ell}(D) \}<p$ $($see Definition \ref{def of cond} for $e_{\ell}(D))$, then we have
\begin{equation*}
Dz_{S} \equiv 0 \mod q  H^1(\Q(S),T). 
\end{equation*}
\end{prop}
\begin{proof}We note that the assumption $w(D)<0$ implies that there exist a prime $\ell \in \mathscr{R}_{E,q}$ dividing $S$ and a derivative $D^{\prime}$ such that 
\begin{equation}\label{d^prime}
D=D^{\prime}N_{\ell}\ \ \ \ \mathrm{Supp}(D^{\prime})=S/\ell, \ \ \ \mathrm{ord}(D^{\prime})=\mathrm{ord}(D).
\end{equation} 

We prove the proposition by induction on the number of primes dividing $S$.
If $S=\ell$ is a prime, then $\ell \in \mathscr{R}_{E,q}$ and $D=N_{\ell}$.
Since $P_{\ell}(1) \equiv 0 \bmod q,$ we have
\begin{equation*}
Dz_{\ell}=N_{\ell}z_{\ell}=P_{\ell}(\Fr_{\ell}^{-1})z_{1} \equiv P_{\ell}(1)z_{1}\equiv0 \mod q.
\end{equation*}

In general, since $w(D)<0$, there exist a prime $\ell \in \mathscr{R}_{E,q}$ dividing $S$ and a derivative  $D^{\prime}$ as in (\ref{d^prime}). 
Then,  we have
\begin{align*}
w(D^{\prime})&=\mathrm{ord}(D^{\prime})-|\{\ell^{\prime} \in \mathscr{R}_{E,q} ; \ \ell^{\prime} \ \text{divides}\ S/\ell \}|\\
&=\mathrm{ord}(D)-|\{\ell^{\prime} \in \mathscr{R}_{E,q} ; \ \ell^{\prime} \ \text{divides}\ S \}|+1\\
 &=w(D)+1 \le 0.
\end{align*}
 We write $S/\ell=\ell_1\cdots \ell_a.$ We show that
 \begin{equation}\label{first invariant}
 (\sigma_{\ell_i}-1)D^{\prime}z_{S/\ell} \equiv 0 \mod q \ \ \ \text{for}\ 1\le i \le a.
 \end{equation}
 It suffices to consider the case $i=1.$
 We write $D^{\prime}=D_{\ell_1}^{(k_1)}\cdots D_{\ell_a}^{(k_a)}$.
In the case where $k_1=0,$ we have $D^{\prime}=N_{\ell_1}D_{\ell_2}^{(k_2)}\cdots D_{\ell_a}^{(k_a)}.$ 
Hence, the claim $(\ref{first invariant})$ is clear.
We may assume that $k_1 \ge 1.$
Since the order of $\sigma_{\ell_1}$ is divisible by $q$ and $0< k_1<p$, Lemma \ref{crucial} implies  that
 \begin{equation}\label{first computation}
(\sigma_{\ell_1}-1)D^{\prime} \equiv -\sigma_{\ell_1}D_{\ell_1}^{(k_1-1)} D_{\ell_{2}}^{(k_2)}\cdots D_{\ell_a}^{(k_a)} \mod q.
\end{equation}
We have
\begin{equation*}
\mathrm{Supp}(D_{\ell_1}^{(k_1-1)} D_{\ell_{2}}^{(k_2)}\cdots D_{\ell_a}^{(k_a)})=S/\ell,\ \ \ \ w(D_{\ell_1}^{(k_1-1)} D_{\ell_{2}}^{(k_2)}\cdots D_{\ell_a}^{(k_a)})=w(D^{\prime})-1<0.
\end{equation*}
Then,  the induction hypothesis implies that 
$D_{\ell_1}^{(k_1-1)} D_{\ell_{2}}^{(k_2)}\cdots D_{\ell_a}^{(k_a)}z_{S/\ell} \equiv 0 \mod q$,
and hence by (\ref{first computation}), we deduce (\ref{first invariant}).

Since each $\Gamma_{\ell_i}$ is generated by $\sigma_{\ell_i}$, the assertion (\ref{first invariant}) implies that
 \begin{equation*}
D^{\prime}z_{S/\ell} \bmod q \in H^0\left(\Gamma_{S/\ell}, H^1(\Q(S/\ell),T)/q\right).
\end{equation*}
Hence,  we have
\begin{equation*}
Dz_S=D^{\prime}N_{\ell}z_S=P_{\ell}(\Fr_{\ell}^{-1})D^{\prime}z_{S/\ell}\equiv P_{\ell}(1) D^{\prime}z_{S/\ell} \equiv 0 \mod q.
\end{equation*} 
\end{proof}

\subsubsection{The proof and an application}

\begin{thm}\label{keythm}
Let $q$ be a power of $p$ and $D$ a Darmon-Kolyvagin derivative with support $S$ satisfying $\max_{\ell|S}\{e_{\ell}(D)\}<p$.
We suppose that $S \in \mathscr{N}_q$ and   for every prime $\ell|S$, $E(\mathbb{F}_{\ell})[p]$ is cyclic, that is, $E(\mathbb{F}_{\ell})[p]=0$ or $E(\mathbb{F}_{\ell})[p] \cong \Z/p\Z$.
If $\mathrm{ord}(D)<r_q\left(H^1_{f,p}(\Q,E[q])\right) $, then 
\begin{equation*}
Dz_{S} \equiv 0 \mod q H^1(\Q(S),T).
\end{equation*}
\end{thm}
We prove it by induction on $w(D).$
Before the proof, we prove some lemmas.
\begin{lem}\label{choice of l}Let $D$ be a Darmon-Kolyvagin derivative with support $S$.
We suppose that $S \in \mathscr{N}_q$ and $Dz_S \bmod q \in H^0\left(\Gamma_S, H^1(\Q(S),T)/q\right).$
We  denote by $\ka$ the inverse image of $Dz_S$ under the isomorphism $H^1(\Q,E[q]) \cong H^0\left(\Gamma_S, H^1(\Q(S),E[q])\right)$.
We put $S^{\prime}=\mathrm{Cond}(D)$.
If $r_q(H^1_{f,pS^{\prime}}(\Q,E[q]))>0,$ then there exists a prime $\ell \in \mathscr{R}$ such that  
\begin{enumerate}
\renewcommand{\labelenumi}{$(\arabic{enumi})$}
\item $\ell\equiv1\mod q$, $\ell$ splits completely in $\Q(S)$, and $E(\Q_{\ell})/q\cong \Z/q\Z$,
\item  $\mathrm{ord}\left(Dz_S, H^1(\Q(S),T)/q\right) = \mathrm{ord}(\loc_{\ell}(\ka),H^1(\Q_{\ell},E[q]))$,
\item the localization map $H^1_{f,pS^{\prime}}(\Q,E[q]) \to E(\Q_{\ell})/q$ is surjective.
\end{enumerate}
\end{lem}
\begin{proof}By  the natural inclusion $H^1(\Q(S),T)/q \hookrightarrow H^1(\Q(S),E[q])$, we have
\begin{equation}\label{zka}
\mathrm{ord}\left(Dz_S, H^1(\Q(S),T)/q\right)=\mathrm{ord}\left(\ka,H^1(\Q,E[q])\right).
\end{equation}
We put $d=\mathrm{ord}\left(\ka,H^1(\Q,E[q])\right).$ 
Since $r_q(H^1_{f,pS^{\prime}}(\Q,E[q]))>0$,  there exists an element $\eta \in H^1_{f,pS^{\prime}}(\Q,E[q])$ of order $q$.
We put $L=\Q(S)(E[q])$ (the composite of $\Q(S)$ and $\Q(E[q])$).
Since  $(S,pN)=1,$ we have $\Gal(L/\Q)=\Gal(\Q(S)/\Q) \times \Gal (\Q(E[q])/\Q)$.
By  \cite[Proposition 6.3 (2)]{bes}, we have
 \begin{equation*}
 H^1(L/\Q(S),E[q]) \cong H^1(\Q(E[q])/\Q,E[q]) =0.
 \end{equation*}
 Then, the restriction map
$H^1(\Q(S),E[q]) \to H^1(L,E[q])$
is injective, and hence by Proposition \ref{trivial torsion}, the restriction map $H^1(\Q,E[q]) \to H^1(L,E[q])$ is injective.
Therefore,  the image of $\ka$ in $H^1(L,E[q])$ is of order $d$ and the image of $\eta$ is of order $q$.
For $\tau \in G_{\Q(\mu_{p^\infty})} $ as in (\ref{tau}), by Lemma \ref{order}, there exists an element $\gamma$ of $G_{L}$ such that
\begin{equation}\label{kappa condition}
\mathrm{ord}(\ka(\gamma\tau),E[q]/(\tau-1)E[q])=d,\ \ \ \mathrm{ord}(\eta(\gamma\tau),E[q]/(\tau-1)E[q])=q.
\end{equation}
We note that $H^1(L,E[q])=\mathrm{Hom}(G_L, E[q]).$ We regard $\ka$, $\eta$ as elements of  $\Hom(G_L,E[q])$
 and put $H=\mathrm{ker}(\ka)\cap\mathrm{ker}(\eta) \subset G_L$.
Let $L^{\prime}$ be a finite Galois extension of $\Q$ containing $\overline{\Q}^{H}$.
For $\sigma \in G_{\Q},$ we denote by $[\sigma]$ the conjugacy class of the image of $\sigma$ in $\Gal(L^{\prime}/\Q).$
By Chebotarev's density theorem, 
there exists a prime $\ell\in \mathscr{R}$ not dividing $pNS$ such that
\begin{equation}\label{conjugacy}
[\mathrm{Fr}_{\ell}] = [\gamma\tau].
\end{equation} 
It remains to show that this $\ell$ satisfies the conditions (1), (2) and (3) above.

(1).\ Since $\gamma\tau = 1 \ \text{in} \ \Gal(\Q(S)(\mu_q)/\Q)$ (recall that $\Q(\mu_q) \subseteq\Q(E[q])$ by the Weil pairing), 
we have $\ell\equiv1\bmod q,$  and $\ell$ splits completely in $\Q(S)$. 
By (\ref{conjugacy}), we have $\Fr_{\ell}=\sigma \tau \sigma^{-1}$ in $\Gal(\Q(E[q])/\Q)$ for some  $\sigma \in \Gal(\Q(E[q]/\Q)$.
Thus, by  $\gamma\tau = \tau \ \text{in} \ \Gal(\Q(E[q])/\Q)$ and (\ref{tau}), we have
\begin{equation*}
H^1(\mathbb{F}_{\ell},E[q])  \cong E[q]/(\mathrm{Fr}_{\ell}-1)E[q] =E[q]/\sigma(\tau-1)E[q] \cong \Z/q\Z,
\end{equation*} where the first isomorphism is given by $f \mapsto f(\Fr_{\ell}) \bmod (\mathrm{Fr}_{\ell}-1)E[q]$.
Since
\begin{equation*}
E(\Q_{\ell})/q = H^1(\mathbb{F}_{\ell},E[q])\ \text{in}\  H^1(\Q_{\ell},E[q]),
\end{equation*}
we deduce that the condition (1) holds.

(2).\ By Proposition \ref{unramified},
  the image $\loc_{\ell}(\ka)$ of $\ka$ in $H^1(\Q_{\ell},E[q])$ belongs to  $ H^1(\mathbb{F}_{\ell},E[q])$.
 By the isomorphism $H^1(\mathbb{F}_{\ell},E[q]) \cong E[q]/(\mathrm{Fr}_{\ell}-1)E[q]$ and (\ref{zka}), it suffices to show that for a lift 
  $\Fr_{\ell} \in G_{\Q}$ of the arithmetic Frobenius at $\ell$,
 \begin{equation}\label{evaluation at fr}
 \mathrm{ord}(\ka(\mathrm{Fr}_{\ell}),E[q]/(\Fr_{\ell}-1)E[q]) =d.
 \end{equation}
 By (\ref{conjugacy}),
 we write $\Fr_{\ell} = \sigma \gamma \tau \sigma^{-1} g \in G_{\Q}$ for some $\sigma \in G_{\Q}$ and $g \in G_{L^{\prime}}$.
 Then, for every $\xi\in H^1(\Q,E[q])$ which is unramified at $\ell$ and satisfies $\xi(g)=0$, we have
 \begin{align*}
 \xi(\Fr_{\ell}) &=\xi(\sigma\gamma\tau\sigma^{-1}g) \overset{\mathrm{(i)}}{=}  \xi(\sigma\gamma\tau\sigma^{-1})=\sigma\xi(\gamma\tau\sigma^{-1})+\xi(\sigma)\\
 &= \sigma(\gamma\tau\xi(\sigma^{-1})+\xi(\gamma\tau))+\xi(\sigma) \overset{\mathrm{(ii)}}{=}  
 \sigma\tau\xi(\sigma^{-1})+\xi(\sigma)+\sigma\xi(\gamma\tau)\\
&=-\sigma\tau\sigma^{-1}\xi(\sigma) +\xi(\sigma) + \sigma\xi(\gamma\tau)=-(\Fr_{\ell}-1)\xi(\sigma)+\sigma\xi(\gamma\tau) \\
&=\sigma\xi(\gamma\tau)\ \text{in}\ E[q]/(\Fr_{\ell}-1)E[q],
 \end{align*}where the equality (i) follows from $\xi(g)=0$, 
 and (ii) follows from $\gamma\in G_{L}.$
 Since
 \begin{equation*}
 (\Fr_{\ell}-1)E[q]=\sigma (\tau-1)E[q],
 \end{equation*}we have
 \begin{equation}\label{order of evaluation}
 \mathrm{ord}\left(\xi(\Fr_{\ell}),E[q]/(\Fr_{\ell}-1)E[q]\right)= \mathrm{ord}\left(\xi(\gamma\tau),E[q]/(\tau-1)E[q]\right).
 \end{equation}
By (\ref{kappa condition}) and (\ref{order of evaluation}) with $\xi=\ka$,  we conclude (\ref{evaluation at fr}).

(3).\  
  By definition,  we have $\loc_{\ell}(\eta) \in H^1(\mathbb{F}_{\ell},E[q]).$
By (\ref{kappa condition}) and (\ref{order of evaluation}) with $\xi=\eta$, we deduce that  $\eta(\mathrm{Fr}_{\ell})$ is of order $q$ in $E[q]/(\Fr_{\ell}-1)E[q],$ and hence the image of $\eta$ in $E(\Q_{\ell})/q$ is of order $q$.
Since  $ E(\Q_{\ell})/q \cong\Z/q\Z$, we prove the assertion (3).
\end{proof}

\begin{lem}\label{cupcup}Under the notation and assumption as in Lemma \ref{choice of l}, we further assume that
  $DD_{\ell}^{(1)}z_S \bmod q \in H^0\left(\Gamma_{S\ell},H^1(\Q(S\ell),T)/q\right).$ 
  Then,
\begin{equation*}
Dz_S \equiv 0 \mod qH^1(\Q(S),T).
\end{equation*}
\end{lem}

\begin{proof}We  denote  by $\ka^{(\ell)}$ the inverse image of $DD^{(1)}_{\ell}z_{S\ell} \bmod q$  under the isomorphism $H^1(\Q,E[q]) \cong H^0\left(\Gamma_{S\ell},H^1(\Q(S\ell),E[q])\right).$
Theorem \ref{ell} (4) implies that
\begin{equation*}
\mathrm{ord}(\loc_{\ell}(\ka^{(\ell)}),H^1(\Q_{\ell},E)[q]) = \mathrm{ord}(\loc_{\ell}(\ka), H^1(\Q_{\ell},E[q])). 
\end{equation*}
By the condition (2) of Lemma \ref{choice of l}, we have
\begin{equation}\label{ell and s}
\mathrm{ord}(\loc_{\ell}(\ka^{(\ell)}),H^1(\Q_{\ell},E)[q]) = \mathrm{ord}\left(Dz_S, H^1(\Q(S),T)/q\right). 
\end{equation}
Hence, we are reduced to showing that the image of $\ka^{(\ell)}$ in $H^1(\Q_{\ell},E)[q]$ is trivial.
For a prime $w$, we denote by $(-,-)_{w}$ the perfect pairing induced by the cup product
\begin{equation*}
E(\Q_{w})/q \times H^1(\Q_{w},E)[q] \to \Z/q\Z.
\end{equation*}
Since the natural map $H^1_{f,pS^{\prime}}(\Q,E[q]) \to E(\Q_{\ell})/q$ is surjective (Lemma \ref{choice of l} (3)),
by taking the Pontryagin dual we have an injective homomorphism
\begin{equation}\label{dual map}
H^1(\Q_{\ell},E)[q] \to \Hom\left(H^1_{f,pS^{\prime}}(\Q,E[q]),\Z/q\Z\right);
\ \ a \mapsto \left(y\mapsto (y,a)_{\ell}\right).
\end{equation}
Hence, it suffices to show that the image of $\ka^{(\ell)}$ is in the kernel of the map above.
Since $p\not=2,$ 
 the Hasse principle shows  that for $x\in H^1_{f,pS^{\prime}}(\Q,E[q])$ 
\begin{equation}\label{hasse}
(x,\ka^{(\ell)})_{\ell}=-\sum_{w \nmid \ell :\text{prime}}(x,\ka^{(\ell)})_{w}.
\end{equation}
If $w\nmid pS^{\prime}\ell,$ then Corollary \ref{f} implies that $\loc_{w}(\ka^{(\ell)}) \in E(\Q_{w})/q,$ and hence
$(x,\ka^{(\ell)})_w=0.$
If $w| pS^{\prime},$
then by the definition of $H^1_{f,pS^{\prime}} (\Q,E[q])$,  we have $(x,\ka^{(\ell)})_w=0$ . 
Therefore, by (\ref{hasse}), we obtain $(x,\ka^{(\ell)})_{\ell}=0$.
Since $x \in H^1_{f,pS^{\prime}}(\Q,E[q])$ is arbitrary, we deduce that the image of $\ka^{(\ell)}$ in $H^1(\Q_{\ell},E)[q]$ is in the kernel of the injection  (\ref{dual map}) and hence it is
 trivial.
\end{proof}

\begin{lem}\label{key lemma}Let $w$ be an integer and $q$ a power of $p$. 
We assume that Theorem \ref{keythm} holds for any Darmon-Kolyvagin derivative whose weight is strictly less than $w$.
Let $D$ be a Darmon-Kolyvagin derivative with support $S$ such that $\max_{\ell|S}\{e_{\ell}(D)\}<p$ and $w(D)=w$.
We suppose that $S \in \mathscr{N}_q$ and for every prime $\ell|S$, $E(\mathbb{F}_{\ell})[p]$ is cyclic.
If $\mathrm{ord}(D)\le r_q\left(H^1_{f,p}(\Q,E[q])\right) $, then 
\begin{equation*}
Dz_{S} \bmod q   \in H^0\left(\Gamma_S,H^1(\Q(S),T)/q\right).
\end{equation*}
\end{lem}
\begin{proof}
We write $S=\ell_1\cdots \ell_s$. 
It suffices to show that
\begin{equation}\label{main rationality}
Dz_{S} \bmod q \in H^0\left(\Gamma_{\ell_i},H^1(\Q(S),T)/q\right)
\end{equation}for each $1\le i \le s.$
It suffices to consider the case $i=1$. 
If $e_{\ell_1}(D)=0$, then we have $D=N_{\ell_1}D^{\prime}$ for some derivative $D^{\prime}$, and hence
 we deduce (\ref{main rationality}).
We  assume that $e_{\ell_1}(D)\ge 1.$
Then, by Lemma \ref{crucial}, we have
\begin{equation*}
(\sigma_{\ell_1}-1)D\equiv -\sigma_{\ell_1}D^{\prime} \mod q\Z[\Gamma_{S}],
\end{equation*}
where $D^{\prime}$ is a derivative such that
$ \text{ord}(D^{\prime}) = \text{ord}(D)-1$ and  $\text{Supp}(D^{\prime})=S$.
 Hence, we have
\begin{equation*}
 w(D^{\prime})= w(D)-1,
 \end{equation*}
which implies that Theorem \ref{keythm} holds for $D^{\prime}$, that is,
 $D^{\prime}z_{S} \equiv 0 \bmod qH^1(\Q(S),T).$ 
Hence, we obtain
\begin{equation*}
(\sigma_{\ell_1}-1)Dz_S\equiv -\sigma_{\ell_1}D^{\prime}z_S \equiv 0 \mod q,
\end{equation*}which shows (\ref{main rationality}).
\end{proof}

\textsc{Proof of Theorem \ref{keythm}}.
We prove the theorem by induction on $w(D)$.
Note that the theorem obviously follows from Proposition \ref{negative} when  $w(D)<0.$
Thus, we may assume that $w:=w(D)\ge 0$ 
and that the theorem holds for any derivative whose weight is strictly  less than $w$. 
Then, by Lemma \ref{key lemma}, we have $Dz_{S} \bmod q \in H^0\left(\Gamma_S,H^1(\Q(S),T)/q\right).
$
We denote by $\ka$ the inverse image of $Dz_{S} \bmod q \in H^0\left(\Gamma_S,H^1(\Q(S),E[q])\right)$ under the isomorphism $H^1(\Q,E[q]) \cong H^0\left(\Gamma_S,H^1(\Q(S),E[q])\right).$

We claim that
\begin{equation*}
r_q\left(H^1_{f,pS^{\prime}}(\Q,E[q])\right)>0.
\end{equation*}
We assume that $r_q\left(H^1_{f,pS^{\prime}}(\Q,E[q])\right)=0$. 
By Lemma \ref{r} and the exact sequence 
\begin{equation*}
0\to H^1_{f,pS^{\prime}}(\Q,E[q]) \to H^1_{f,p}(\Q,E[q])\to \oplus_{\ell|S^{\prime}}E(\Q_{\ell})/q,
\end{equation*}we have
\begin{equation*}
r_q\left(H^1_{f,p}(\Q,E[q])  \right)\le r_q\left(H^1_{f,pS^{\prime}}(\Q,E[q])  \right)+r_p(\oplus_{\ell|S^{\prime}}E(\Q_{\ell})/p),
\end{equation*}
and hence by the assumption,
\begin{equation*}
\mathrm{ord}(D)<\sum_{\ell|S^{\prime}}r_p(E(\Q_{\ell})/p).
\end{equation*}
For a prime $\ell \nmid pN,$ we have
\begin{equation*}
E(\Q_{\ell})/p \cong  E(\Q_{\ell})[p] \cong E(\mathbb{F}_{\ell})[p],
\end{equation*}where the first (non-canonical) isomorphism is due to the structure theorem for finite abelian groups and to that $E(\Q_{\ell}) \cong \Z_{\ell} \oplus E(\Q_{\ell})_{\tors}$.
 For a prime $\ell|S$, $E(\mathbb{F}_{\ell})[p]$ is assumed to be cyclic, and then we have $r_{p}(E(\Q_{\ell})/p)\le1$.
Hence,
\begin{equation*}
\mathrm{ord}(D)< \sum_{\ell|S^{\prime}}1.
\end{equation*}However, by the definition of  $S^{\prime}=\mathrm{Cond}(D),$ we have
$\sum_{\ell|S^{\prime}}1 \le \mathrm{ord}(D).$ 
Then, we have a contradiction and hence  $r_q\left(H^1_{f,pS^{\prime}}(\Q,E[q])\right)>0.$
 
By the claim above,  there exists a prime $\ell$ satisfying the conditions (1), (2) and (3) in Lemma \ref{choice of l} for $Dz_S$.
 Since $\mathrm{ord}(DD_{\ell}^{(1)}) \le r_q\left(H^1_{f,p}(\Q,E[q])\right),$
  by Lemma \ref{key lemma} we have
 \begin{equation*}
 DD_{\ell}^{(1)}z_{S\ell} \bmod q \in H^0\left(\Gamma_{S\ell},H^1(\Q(S\ell),T)/q\right).
 \end{equation*}Hence, Lemma \ref{cupcup} implies that  $Dz_S \equiv 0 \mod q$.
\qed

\begin{rem}\label{gap} If the image of  $\ka^{(\ell)}$ in $H^1(\Q_{p},E[q])$ always belonged to $E(\Q_{p})/q$  as   in \cite[Theorem 4.9]{dar92}, our proof would work under the assumption that  $\mathrm{ord}(D) < r_q(\Sel(\Q,E[q]))$. 
\end{rem}

We put $\displaystyle \mathfrak{r}_{\min}=\min_{n\ge 1}\left\{r_{p^n}\left(H^1_{f,p}(\Q,E[p^n])\right)\right\}.$
\begin{cor}\label{general divisibility}Let $S$ be an element of $\mathscr{N}$ such that  for each prime $\ell |S,$ $E(\mathbb{F}_{\ell})[p]$ is cyclic.
Then, we have
\begin{equation*}
\sum_{\sigma\in\Gamma_S}z_S^{\sigma^{-1}}\otimes\sigma \in H^1(\Q(S),T)\otimes I_{\Gamma_S}^{\min\{\mathfrak{r}_{\min},p\}}.
\end{equation*}
\end{cor}

\begin{proof}We may assume that $\mathfrak{r}_{\min}\ge 1$.
To apply Lemma \ref{fromcongtovanishing} for $H^1(\Q(S),T)$ and $z_S$, we take
a derivative $D$  such that $\mathrm{Supp}(D)=S$ and $\mathrm{ord}(D) < \min\{\mathfrak{r}_{\min},p\}$.
We denote by $S^{\prime}$ the conductor of $D$, and then 
$D=D^{\prime}N_{\frac{S}{S^{\prime}}}$,
where the derivative $D^{\prime}$ satisfies 
\begin{equation*}
\text{Supp}(D^{\prime})=\text{Cond}(D^{\prime})=S^{\prime},\ \ \  n(D^{\prime})=n(D),\ \ \ \mathrm{ord}(D^{\prime}) = \mathrm{ord}(D).
\end{equation*}
Therefore,
\begin{equation*}
Dz_S=\left(\prod_{\ell | (S/S^{\prime})}P_{\ell}(\Fr_{\ell}^{-1})\right)D^{\prime}z_{S^{\prime}},
\end{equation*}where $\ell$ ranges over all the primes dividing $S/S^{\prime}.$ 
By definition, if we put $q=n(D^{\prime})$, which is a power of $p$,
  then  $S^{\prime} \in \mathscr{N}_{q}.$
Since
\begin{equation*}
\mathrm{ord}(D^{\prime})=\mathrm{ord}(D) <\mathfrak{r}_{\min}\le r_q\left(H^1_{f,p}(\Q,E[q])\right), \ \ \  \max_{\ell|S^{\prime}}\{e_{\ell}(D^{\prime})\}\le \mathrm{\ord}(D^{\prime})<p,
\end{equation*}
Theorem \ref{keythm} implies that  
$D^{\prime}z_{S^{\prime}} \equiv 0 \bmod q,$
and hence
$ Dz_S \equiv 0  \bmod q.$
Consequently, Lemma \ref{fromcongtovanishing} shows that
\begin{equation}\label{not finished vanishing order}
\sum_{\sigma\in\Gamma_S}z_S^{\sigma^{-1}}\otimes\sigma -N_S z_S\otimes 1\in H^1(\Q(S),T)\otimes I_{\Gamma_S}^{\min\{\mathfrak{r}_{\min},p\}}.
\end{equation}
Since the Galois representation $G_{\Q} \to \mathrm{Aut}_{\Zp}(T)$ is assumed to be surjective, $E(\Q)[p^\infty] =0$, and then the natural map $H^1(\Q,E[p^n]) \to H^1(\Q,E[p^{\infty}])$ is injective for all $n\ge 1$. Hence,   the inductive limit $\displaystyle H^1_{f,p}(\Q,E[p^{\infty}]):=\varinjlim_{n}H^1_{f,p}(\Q,E[p^n])$ is \textit{not} finite, since $\mathfrak{r}_{\min}\ge 1.$
By \cite[Theorem 2.2.3]{rub} (our $H^1_{f,p}(\Q,E[p^{\infty}])$ coincides with $\mathcal{S}_{\Sigma_p}(\Q,E[p^\infty])$ in \cite{rub}), we have
$z_1 =0$, and
then $N_Sz_S = \prod_{\ell|S}P_{\ell}(1)z_1=0.$
From this and (\ref{not finished vanishing order}), we complete the proof.
\end{proof}

\subsubsection{A modification of the theorem}

\begin{thm}\label{general}
Let $q$ be a power of $p$.
Let  $D$ be a Darmon-Kolyvagin derivative with support $S$ such that  $\max_{\ell|S}\{e_{\ell}(D)\}<p$. 
We suppose that $S \in \mathscr{N}_q$ and for each prime $\ell $ dividing $S$, $E(\mathbb{F}_{\ell})[q]$ is isomorphic to $\Z/q\Z$ or 0.
We put $S^{\prime}=\mathrm{Cond}(D)$ and recall that $A_q(S^{\prime}) := \oplus_{\ell|S^{\prime}}E(\Q_{\ell})/q$.
If $\mathrm{ord}(D)<r_q\left(H^1_{f,pS^{\prime}}(\Q,E[q])\right)+r_{p}(A_q(S^{\prime})) $, then we have
\begin{equation*}
Dz_{S} \equiv 0 \mod qH^1(\Q(S),T).
\end{equation*}
\end{thm}

\begin{rem}\label{rank relation}
 By Lemma \ref{r}, we have 
\begin{equation*}
r_q(H^1_{f,p}(\Q,E[q])) \le r_q\left(H^1_{f,pS^{\prime}}(\Q,E[q])\right)+r_{p}(A_q(S^{\prime})).
\end{equation*}In particular, when $q=p,$ Theorem \ref{general} implies Theorem \ref{keythm}.
\end{rem}

\begin{lem}\label{complicated induction}
Let $w$ be an integer  and $q$ a power of $p$. We assume that Theorem \ref{general} holds for any derivative whose weight is strictly less than $w$.
Let  $D, S$ and $S^{\prime}$ be as in Theorem \ref{general}.
We suppose that  $w(D)=w$. 
If $\mathrm{ord}(D)\le r_q\left(H^1_{f,pS^{\prime}}(\Q,E[q])\right)+r_{p}(A_q(S^{\prime})) $, then
\begin{equation*}
Dz_{S} \bmod q  \in H^0\left(\Gamma_S,H^1(\Q(S),T)/q\right).
\end{equation*}
\end{lem}

\begin{proof}
We simply write  $H^1_{f,*}= H^1_{f,*}(\Q,E[q]).$
Let $S=\ell_1\cdots\ell_s $ be the prime factorization. We write $D=D_{\ell_1}^{(k_1)}\cdots D_{\ell_s}^{(k_s)},$ and then for each $1 \le i \le s$, one of the following assertions holds:
\begin{enumerate}
\item[(a)] $k_i=0$. 
\item[(b)] $k_i\ge2$.
\item[(c)] $k_i =1$, $\ell_i \in \mathscr{R}_{q}\setminus \mathscr{R}_{E,q}$, and hence $E(\Q_{\ell_i})/q=0.$
\item[(d)] $k_i=1$, $\ell_i \in \mathscr{R}_{E,q}$, and hence $E(\Q_{\ell_i})/q \cong \Z/q\Z.$
\end{enumerate}
It suffices to show that for each  $1\le i \le s$,
\begin{equation}\label{modified rationality}
Dz_{S} \in H^0\left(\Gamma_{\ell_i},H^1(\Q(S),T)/q\right).
\end{equation}

\noindent\textbf{Step 1.}
For each $i$ satisfying  (a), (b) or (c) above, the assertion (\ref{modified rationality}) holds.

We only need to consider the case $i=1.$
If $k_1=0$, then we have $D\in N_{\ell_1}\Z[\Gamma_S]$, and hence $Dz_{S} \in H^0\left(\Gamma_{\ell_1},H^1(\Q(S),T)/q\right).$
 Hence, we may assume that $k_1 \ge 1.$
Then by Lemma \ref{crucial},
\begin{equation}\label{modified case 2}
(\sigma_{\ell_1}-1)D\equiv -\sigma_{\ell_1}D_{\ell_1}^{(k_1-1)} D_{\ell_2}^{(k_2)}\cdots D_{\ell_s}^{(k_s)} \mod q \Z[\Gamma_{S}].
\end{equation}We put $D^{\prime} = D_{\ell_1}^{(k_1-1)}  D_{\ell_2}^{(k_2)}\cdots D_{\ell_s}^{(k_s)}.$
Then, 
\begin{equation*}
\mathrm{Supp}(D^{\prime})=S,\ \ \ \  \mathrm{ord}(D^{\prime})=\mathrm{ord}(D)-1,
\end{equation*}
and hence
\begin{equation*}
w(D^{\prime})=w(D)-1.
\end{equation*}
Since $\sigma_{\ell_1}$ generates $\Gamma_{\ell_1}$, by (\ref{modified case 2}), it suffices to show that
\begin{equation*}
D^{\prime}z_S \equiv 0 \mod q.
\end{equation*}
We consider the cases  (b) and (c).

\textit{Case} (b).\ In this case, we have $\mathrm{Cond}(D^{\prime})=S^{\prime}.$
 We recall that
 \begin{equation*}
 \mathrm{ord}(D^{\prime})<\mathrm{ord}(D) \le  r_q(H^1_{f,pS^{\prime}})+r_p(A_q(S^{\prime})).
 \end{equation*}
Then, Theorem \ref{general} holds for $D^{\prime}$, that is,
$D^{\prime}z_S \equiv 0 \mod q.$

\textit{Case}  (c).\ In this case, we have $\mathrm{Cond}(D^{\prime})=S^{\prime}/\ell_1.$
Since $E(\Q_{\ell_1})/q=0$, we have 
\begin{equation*}
r_q(H^1_{f,pS^{\prime}/\ell_1})=r_q(H^1_{f,pS^{\prime}}),\ \ \  r_p(A_q(S^{\prime}/\ell_1))=r_p(A_q(S^{\prime})).
\end{equation*}
Then, we have
\begin{equation*}
\mathrm{ord} (D^{\prime}) <  r_q(H^1_{f,pS^{\prime}/\ell_1})+r_p(A_q(S^{\prime}/\ell_1)).
\end{equation*}
Hence, Theorem \ref{general} holds for $D^{\prime}$, that is, $D^{\prime}z_{S}\equiv 0 \mod q.$

\noindent\textbf{Step 2.}
We prove the  lemma by induction on the number $n$ of primes satisfying (d).
Without loss of generality, we may write $S=\ell_1\cdots \ell_s,$ where  $\ell_1,\ell_2, \ldots,\ell_n$ satisfy (d) and $\ell_{n+1},\ell_{n+2},\ldots, \ell_{s}$ satisfy (a), (b) or (c).

 \textit{The case} $n=0$.
This case follows from  Step 1. 

  \textit{The case} $n\ge1.$ By Step 1, we are reduced to showing that  for $1\le i \le n$
  \begin{equation*}
  Dz_{S} \bmod q  \in H^0\left(\Gamma_{\ell_i},H^1(\Q(S),T)/q\right).
  \end{equation*}
It suffices to consider the case $i=1$.
Let $S_1=S/\ell_1$ and $D_{S_1}= D_{\ell_2}^{(k_2)} \cdots D_{\ell_s}^{(k_s)}.$
 Then, we have
\begin{equation}\label{ell_11}
(\sigma_{\ell_1}-1)Dz_{S} \equiv -N_{\ell_1}D_{S_1}z_{S} \equiv -P_{\ell_1}(\Fr_{\ell_1}^{-1})D_{S_1}z_{S_1} \mod q.
\end{equation}
By Lemma \ref{variation of r},
\begin{align*}
\mathrm{ord}(D_{S_1}) = \mathrm{ord}(D) -1 &\le r_q(H^1_{f,pS^{\prime}})+r_p(A(S^{\prime}))-1\\
&\le r_q(H^1_{f,pS^{\prime}/\ell_1})+r_p(A(S^{\prime}/\ell_1)).
\end{align*}We recall that $\mathrm{Cond}(D_{S_1})=S^{\prime}/\ell_1.$
Since $\ell_1 \in \mathscr{R}_{E,q}$, we have $w(D_{S_1})=w(D)=w.$
Therefore, we may apply the induction hypothesis on $n$ to $D_{S_1}$, and  then obtain \begin{equation*}\label{trivial for S_1}
D_{S_1}z_{S_1} \bmod q \in H^0\left(\Gamma_{S_1},H^1(\Q(S_1),T)/q\right).
\end{equation*}Hence, we have
\begin{equation*}
P_{\ell_1}(\Fr_{\ell_1}^{-1})D_{S_1}z_{S_1} \equiv P_{\ell_1}(1)D_{S_1}z_{S_1} \equiv 0 \mod q,
\end{equation*}and then by (\ref{ell_11}), we  complete Step 2.
\end{proof}

\textsc{Proof of Theorem \ref{general}.}
As in the proof of Theorem \ref{keythm},  Theorem \ref{general} is proved by induction on $w(D).$
By Lemma \ref{complicated induction}, $Dz_{S} \bmod q \in H^0\left(\Gamma_S,H^1(\Q(S),T)/q\right).$
Let $\ka$ be the inverse image of $Dz_S \bmod q$ under the isomorphism $H^1(\Q,E[q]) \cong H^0\left(\Gamma_S,H^1(\Q(S),E[q])\right).$
Since $\mathrm{ord}(D)< r_q\left(H^1_{f,pS^{\prime}}(\Q,E[q])\right)+r_{p}(A_q(S^{\prime}))$, by the same argument as that in the proof of Theorem {\ref{keythm}}, we  have $r_q(H^1_{f,pS^{\prime}}(\Q,E[q]))>0$.
Hence, by Lemma \ref{choice of l} there exists a prime $\ell$ the conditions (1), (2) and (3) in Lemma \ref{choice of l} for $Dz_S$.
Using (\ref{usually for l}), we have 
\begin{equation*}
\mathrm{ord}(DD_{\ell}^{(1)}) \le r_q\left(H^1_{f,pS^{\prime}\ell}(\Q,E[q])\right)+r_{p}(A_q(S^{\prime}\ell)).
\end{equation*}
 Since $\mathrm{Cond}(DD_{\ell}^{(1)})=S^{\prime}\ell$, by Lemma \ref{complicated induction} we have
 \begin{equation*}
 DD_{\ell}^{(1)}z_{S\ell} \bmod q \in H^0\left(\Gamma_{S\ell},H^1(\Q(S\ell),T)/q\right).
 \end{equation*}Then, Lemma \ref{cupcup} implies that $Dz_S \equiv 0 \mod q$.
\qed

\subsection{Local behavior of derivatives of Euler systems at $p$}\label{local behavior at p}
In this subsection, we relate the localization of certain derivatives of Euler systems at $p$ with arithmetic invariants such as the Tate-Shafarevich group (Corollary \ref{toward leading term}).

We put $r_E=\mathrm{rank}(E(\Q))$ and denote by $\Sha$ the Tate-Shafarevich group of $E$ over $\Q$.
For a positive integer $S$ and $A=T_p(E), V_p(E)$ or $E[p]$,
we put 
\begin{equation*}
H^1(\Q(S)\otimes\Qp,A)=\oplus_{\lambda|p}H^1(\Q(S)_{\lambda},A),
\end{equation*}where $\lambda$ ranges over all the primes of $\Q(S)$ dividing $p$, and $\Q(S)_{\lambda}$ denotes the  completion at $\lambda$. 
We denote by $H^1_f(\Q(S)\otimes\Qp,A)$ the image of the Kummer map, and
define
\begin{equation*}
H^1_{/f}(\Q(S)\otimes\Qp,A)=\frac{H^1(\Q(S)\otimes\Qp,A)}{H^1_f(\Q(S)\otimes\Qp,A)}.
\end{equation*}

For $\eta \in H^1(\Q(S),E[p]),$ we denote by $\loc_p(\eta)$  the image of $\eta$ in $H^1(\Q(S)\otimes \Qp,E[p]).$
\begin{thm}\label{s}We assume that $E(\Qp)/p \cong \Z/p\Z$.
Let $D$ be a Darmon-Kolyvagin derivative with support $S$ such that $\max_{\ell|S}\{e_{\ell}(D)\}<p$. 
Suppose  that $S\in\mathscr{N}_p$, and for each prime $\ell |S$, $E(\mathbb{F}_{\ell})[p]$ is cyclic.
We put $S^{\prime}=\mathrm{Cond}(D)$.
If $\mathrm{ord}(D)<r_p\left(H^1_{f,S^{\prime}}(\Q,E[p])\right) + r_p(A_p(S^{\prime}))$,
 then the following assertions hold.
\begin{enumerate}
\item $Dz_S \bmod p \in H^0\left(\Gamma_S,H^1(\Q(S),T)/p\right).$
\item If we denote by $\ka  \in H^1(\Q,E[p])$ the inverse image of $Dz_S \bmod p$ under the isomorphism 
$H^1(\Q,E[p]) \cong H^0\left(\Gamma_S,H^1(\Q(S),E[p])\right)$, then we have 
\begin{equation*}
\loc_p(\ka) \in  H^1_{f}(\Qp,E[p]).
\end{equation*}
\end{enumerate}
\end{thm}
\begin{proof}We write $H^1_{f,*}=H^1_{f,*}(\Q,E[p]).$
Since $E(\Q_p)/p \cong \Z/p\Z$, by the exact sequence
\begin{equation}\label{ps and s}
0\to H^1_{f,pS^{\prime}}\to H^1_{f,S^{\prime}} \to E(\Qp)/p,
\end{equation} we have
$ r_p\left(H^1_{f,S^{\prime}}\right) \le r_p\left(H^1_{f,pS^{\prime}}\right) +1$,
and then the assertion (1) follows from Lemma \ref{complicated induction}.

We prove the assertion (2).
If $r_p\left(H^1_{f,S^{\prime}}\right) =  r_p\left(H^1_{f,pS^{\prime}}\right)$, then Theorem \ref{general} implies that $Dz_S \equiv 0 \bmod pH^1(\Q(S),T),$ and hence $\ka=0$.
 We assume  that 
$r_p\left(H^1_{f,S^{\prime}}\right) =  r_p\left(H^1_{f,pS^{\prime}}\right)+1$.
Then, by  (\ref{ps and s}), the map $H^1_{f,S^{\prime}}\to E(\Qp)/p$ is surjective.
Since the pairing
\begin{equation*}
(-,-)_p: E(\Qp)/p \times H^1_{/f}(\Qp,E[p]) \to \Z/p\Z
\end{equation*}is perfect,
it suffices to show that
\begin{equation*}
(c,\ka)_p=0\ \ \ \ \mathrm{for} \  c \in E(\Qp)/p.
\end{equation*} 
We take an element  $c \in E(\Qp)/p$. 
Since $H^1_{f,S^{\prime}}\to E(\Qp)/p$ is surjective,
 there exists an element $\eta \in H^1_{f,S^{\prime}}$ whose localization at $p$ coincides  with $c$.
Then, by the  Hasse principle,
\begin{equation*}
(c,\ka)_p = (\eta,\ka)_p=-\sum_{w\nmid p: \mathrm{prime}}(\eta,\ka)_w.
\end{equation*}
By the definition of $H^1_{f,S^{\prime}}(\Q,E[p])$ and Corollary \ref{f} for $\ka$, 
we have $(\eta,\ka)_w=0$ for $w\nmid p$, and hence  $(c,\ka)_p=0.$
\end{proof}

 The following plays an important role in the proof of Theorem \ref{theorem on leading term}.
\begin{cor}\label{toward leading term}We assume that $E(\Q_p)/p \cong \Z/p\Z$ and $p\ge r_E$. 
Let $D$ be a Darmon-Kolyvagin derivative with support $S$. 
We suppose that  $S\in\mathscr{N}_p$ and $E(\mathbb{F}_{\ell})[p]$ is cyclic for each prime $\ell|S$. 
If $\mathrm{ord}(D)=r_E$ and $\loc_p(Dz_S \bmod p) \notin H^1_{f}(\Q(S)\otimes\Qp ,E[p] ),$ then we have
\begin{enumerate}
\item $\Sha[p] =0,$
\item the natural map $E(\Q)/p \to \oplus_{\ell|S}E(\Q_{\ell})/p $ is surjective.
\end{enumerate}
\end{cor}
\begin{proof}We put $S^{\prime}=\mathrm{Cond}(D)$.
Since $\loc_p(Dz_S) \notin  H^1_{f}(\Q(S)\otimes\Qp ,E[p] ),$ Theorem \ref{s} implies that
 \begin{equation*}
r_E \ge r_{p}(H^1_{f,S^{\prime}})+r_{p}(A(S^{\prime})).
\end{equation*}On the other hand, by Lemma \ref{r},  
\begin{equation*}
r_{p}(H^1_{f,S^{\prime}})+r_{p}(A(S^{\prime}))\ge r_p(\Sel(\Q,E[p])),
\end{equation*}and then we have
\begin{equation}\label{rank and s prime}
r_E = r_{p}(H^1_{f,S^{\prime}})+r_{p}(A(S^{\prime})) = r_p(\Sel(\Q,E[p])).
\end{equation}
This implies that $\Sha[p]=0$, and the sequence
\begin{equation*}
0\to H^1_{f,S^{\prime}} \to E(\Q)/p \to \oplus_{\ell|S^{\prime}}E(\Q_{\ell})/p \to 0
\end{equation*}is exact.
Then, it remains to show that for each prime $\ell$ dividing $S/S^{\prime}$, $E(\Q_{\ell})/p=0$. 
We assume that  $E(\Q_{\ell})/p \cong \Z/p\Z$ for some $\ell$ dividing $S/S^{\prime}$.
Since $\ell\nmid S^{\prime}$,
 we have $D=N_{\ell}D^{\prime}$, where $D^{\prime}$ is a derivative such that $\mathrm{Supp}(D^{\prime})=S/\ell$ and $\mathrm{ord}(D^{\prime})=r_E$. 
We claim  that 
\begin{equation}\label{s/ell}
\loc_p(D^{\prime}z_{S/\ell} \bmod p) \in H^0(\Gamma_{S/\ell},H^1_{/f}(\Q(S/\ell)\otimes\Qp,E[p])).
\end{equation}
In order to prove this, we take a prime  $\ell^{\prime}$ dividing $S/\ell$. 
If $e_{\ell^{\prime}}(D^{\prime})=0,$ then $D^{\prime} \in N_{\ell^{\prime}}\Z[\Gamma_{S/\ell}],$ and hence
we have $\loc_p(D^{\prime}z_{S/\ell} \bmod p) \in H^0(\Gamma_{\ell^{\prime}},H^1_{/f}(\Q(S/\ell)\otimes\Qp,E[p]))$.
We assume that $e_{\ell^{\prime}}(D^{\prime})\ge 1.$
Then, we have 
$(\sigma_{\ell^{\prime}}-1)D^{\prime}\equiv -\sigma_{\ell^{\prime}}D^{\prime\prime} \bmod p,$ where $D^{\prime\prime}$ satisfies $\mathrm{ord}(D^{\prime\prime}) = r_E-1.$
If we put $S^{\prime\prime}=\mathrm{Cond}(D^{\prime\prime}),$ then by Lemma \ref{r},
we have
\begin{equation*}
r_E-1 <r_p(\Sel(\Q,E[p])) \le r_p(H^1_{f,S^{\prime\prime}})+ r_p(A_p(S^{\prime\prime})).
\end{equation*}Hence, Theorem \ref{s} implies that
$\loc_p( D^{\prime\prime}z_{S/\ell} \bmod p) \in   H^1_{f}(\Q(S/\ell)\otimes \Qp, E[p]),$
 and then
\begin{equation*}
\loc_p(D^{\prime}z_{S/\ell} \bmod p) \in H^0(\Gamma_{\ell^{\prime}},H^1_{/f}(\Q(S/\ell)\otimes\Qp,E[p])).
\end{equation*}Therefore, we deduce (\ref{s/ell}).

By (\ref{s/ell}), in $H^1_{/f}(\Q(S)\otimes\Qp,E[p]))$, we have
\begin{equation*}
\loc_p(Dz_S) = \loc_p(N_{\ell}D^{\prime} z_S) = P_{\ell}(\Fr_{\ell}^{-1})\loc_p(D^{\prime}z_{S/\ell}) = P_{\ell}(1)\loc_p(D^{\prime}z_{S/\ell})=0.
\end{equation*}
Then, we have a contradiction.
\end{proof}

\subsection{Rational points from derivatives of Euler systems}\label{chapter: construction of rational points}In this subsection,  we show that if a certain derivative of an Euler system is not divisible by $p$, then it comes from a $\Q$-rational point of $E$.
We assume that $E(\Q_p)/p \cong \Z/p\Z$ and the natural map
 $E(\Q)/p \to E(\Qp)/p$ is surjective.
In particular, the localization map $\Sel(\Q,E[p]) \to E(\Qp)/p$ is surjective, 
and hence 
\begin{equation*}
r_p\left(H^1_{f,p}(\Q,E[p])\right)=r_p\left(\Sel(\Q,E[p])\right) -1.
\end{equation*}
We put  
$C_p = \mathrm{ker}\left(E(\Q)/p \to E(\Q_p)/p\right)$.
Then, we have 
\begin{equation}\label{rank of cp}
r_p(C_p)=r_E-1.
\end{equation}
By  applying the snake lemma to the commutative diagram
\begin{equation*}
\begin{CD}
0 @>>> E(\Q)/p @>>> \Sel(\Q,E[p]) @>>>\Sha[p] @>>>0  \\ & &
 @VVV                                          @VVV                                            @VVV \\
0 @>>> E(\Qp)/p @>>> E(\Qp)/p @>>> 0 @>>>0,\\  
  \end{CD}
\end{equation*}we have an exact sequence
\begin{equation}\label{fine sha}
0\to C_p \to H^1_{f,p}(\Q,E[p]) \to \Sha[p] \to 0.
\end{equation}

\begin{thm}\label{sel}We assume that $E(\Q_p)/p \cong \Z/p\Z$ and the map $E(\Q)/p \to E(\Qp)/p$ is surjective.
Let $D$ be a Darmon-Kolyvagin derivative with support $S$ and $\max_{\ell|S}\{e_{\ell}(D)\}<p$. 
We suppose that $S\in\mathscr{N}_p$, and for each prime $\ell|S$, $E(\mathbb{F}_{\ell})[p]$ is cyclic.
 If  $\mathrm{ord}(D)=r_E-1$ and $Dz_S \not\equiv 0 \bmod pH^1(\Q(S),T)$, then the following assertions hold.
\begin{enumerate}
\item $\Sha[p]=0$.
\item The localization  map $H^1_{f,p}(\Q,E[p]) \to \oplus_{\ell|S}E(\Q_{\ell})/p$ is surjective.
\item $Dz_S \bmod p \in H^0\left(\Gamma_S, H^1(\Q(S),T)/p\right).$ 
\item If $\ka \in H^1(\Q,E[p])$ denotes the inverse image of $Dz_S \bmod p$ under the isomorphism 
$H^1(\Q,E[p]) \cong H^0\left(\Gamma_S, H^1(\Q(S),E[p])\right)$,
then we have
\begin{equation*}
\ka \in E(\Q)/p.
\end{equation*}
\end{enumerate}
\end{thm}
\begin{proof}(1).\ 
Since $Dz_S \not\equiv 0 \bmod p$, Theorem \ref{keythm} implies that
$r_E-1 \ge r_p(H^1_{f,p}).$ 
Hence, by (\ref{rank of cp}) and (\ref{fine sha}), we have $r_p(H^1_{f,p})=r_E-1$ and $\Sha[p]=0$.

(2).\ Since $\mathrm{ord}(D)=r_p(H^1_{f,p})$ and $Dz_S \not\equiv 0 \bmod p$,  
Theorem \ref{general} implies that 
\begin{equation*}
r_p(H^1_{f,p})\ge r_p(H^1_{f,pS^{\prime}})+r_p(A_p(S^{\prime})).
\end{equation*}
Hence, the sequence
\begin{equation}\label{trivial j}
0\to H^1_{f,pS^{\prime}}\to H^1_{f,p} \to \oplus_{\ell|S^{\prime}}E(\Q_{\ell})/p \to 0
\end{equation}is exact. 
In order to deduce the assertion (2),
 it suffices to show that for each prime $\ell$ dividing $S/S^{\prime}$, we have $E(\Q_{\ell})/p=0$.
We assume that $E(\Q_{\ell})/p\cong \Z/p\Z$ for some prime $\ell$ dividing $S/S^{\prime}$.
Since $\ell\nmid S^{\prime}$, we have
$D=N_{\ell}D^{\prime}$,
where $D^{\prime}$ is a derivative such that 
\begin{equation*}
\mathrm{Supp}(D^{\prime})=S/\ell, \ \ \ \mathrm{Cond}(D^{\prime})=\mathrm{Cond}(D), \ \ \ \mathrm{ord}(D^{\prime})=\mathrm{ord}(D)=r_E-1.
\end{equation*} 
Lemma \ref{key lemma} implies that 
 $D^{\prime}z_{S/\ell} \in H^0\left(\Gamma_{S/\ell},H^1(\Q(S/\ell),T)/p\right),$
and hence, we have
\begin{equation*}
Dz_S = N_{\ell}D^{\prime}z_{S}=P_{\ell}(\Fr_{\ell}^{-1})D^{\prime}z_{S/\ell} \equiv P_{\ell}(1)D^{\prime}z_{S/\ell} \equiv 0 \mod p.
\end{equation*}
Hence, we obtain a contradiction.

(3).\  The assertion (3) follows from Lemma \ref{key lemma}.

(4).\ 
We first show that $\ka \in \Sel(\Q,E[p]).$
By Corollary \ref{f}, we are reduced to showing that for each prime $\ell | pS^{\prime}$, we have $\loc_{\ell}(\ka) \in E(\Q_{\ell})/p$.
By taking the Pontryagin dual of the sequence (\ref{trivial j}), the map
\begin{equation*}
\varphi: \bigoplus_{\ell|S^{\prime}}H^1(\Q_{\ell},E)[p] \to H^1_{f,p}(\Q,E[p])^{\vee}; \ (g_{\ell})_{\ell} 
\mapsto \left(\eta\mapsto\sum_{\ell|S^{\prime}}(g_{\ell},\eta)_{\ell}\right)
\end{equation*}is injective.
We claim that the image of $\ka$ in $\oplus_{\ell|S^{\prime}}H^1(\Q_{\ell},E)[p]$ belongs to the kernel of the map above.
Indeed, if we take an element $\eta\in H^1_{f,p}(\Q,E[p])$,
then by the Hasse principle, 
\begin{equation*}
\sum_{\ell|S^{\prime}}(\ka,\eta)_{\ell} =-\sum_{v\nmid S^{\prime}}(\ka,\eta)_{v}=-(\ka,\eta)_p=0,
\end{equation*}where the second equality follows from Corollary \ref{f}, and the last equality follows from the definition of $H^1_{f,p}(\Q,E[p])$.
 Since the map $\varphi$ is injective, 
  $\loc_{\ell}(\ka) \in E(\Q_{\ell})/p$ for all $\ell|S^{\prime}.$
In addition, by Theorem \ref{s}, we have
$\loc_p(\ka) \in E(\Q_p)/p$.
Then, we deduce that $\ka \in \Sel(\Q,E[p])$. 
Hence by the assertion (1), we have $\ka \in E(\Q)/p$.
\end{proof}
\begin{rem}\begin{enumerate}
\item For Heegner points, a similar result was obtained in \cite[Proposition 5.10]{dar92}, where $K$-rational points are considered for imaginary quadratic fields $K$.
\item Zhang \cite{zha14} recently proved a conjecture of Kolyvagin, which asserts indivisibility of Kolyvagin derivatives of Heegner points.
\end{enumerate}
\end{rem}



\section{Local study of Mazur-Tate elements}\label{section: kato and mazur-tate}
The aim of this section is to show that if we extend coefficients to $\Zp$, then the order of vanishing of Mazur-Tate elements is greater than or equal to the corank of the Selmer group  (see Theorem \ref{p-vanishing} for the precise statement).
We fix a global minimal Weierstrass model of $E$ over $\Z$, and denote by $\omega$ the N\'eron differential.
Then, let $\Omega^{\pm}$ be the period as in Section \ref{chapter of refined BSD}.

\subsection{Preliminaries on group rings}
\subsubsection{Local property}
Let $p$ be a prime.
For a finite abelian group $G$, 
we denote by $I_G$ the augmentation ideal of $\Zp[G]$.


\begin{lem}\label{order prime to p}
For an element $\sigma  \in G$ whose order is relatively prime to $p$,
we have $\sigma-1 \in I_{G}^t$ for all $t\ge 1$. 
In particular, if $p\nmid |G|$, then $I_{G}=I_{G}^2= I_{G}^3=\cdots.$  
\end{lem}
\begin{proof}This is \cite[Lemma 3.4]{dar92}.
\end{proof}
\begin{lem}\label{p-prime}
Suppose that we are given a decomposition $G=K\times H$ with  $p\nmid  |H|.$
Let $\alpha$ be an element of $\Zp[G]$. Let  $\alpha_K$ denote the image of $\alpha$ under the map $\Zp[G] \to \Zp[K]$ induced by the projection $G\to K.$
If $\alpha_K\in  I_{K}^t$ for some $t\ge 1$, then $\alpha \in  I_{G}^t$.
\end{lem}
\begin{proof}By the natural inclusion $\Zp[K]\hookrightarrow \Zp[G]$,  
we regard $\alpha_K$ as an element of $\Zp[G]$.
Then, we have 
\begin{equation*}
\alpha-\alpha_K \in \mathrm{ker}(\Zp[G]\to \Zp[K])=\Zp[K]\otimes_{\Zp}I_{H}.
\end{equation*}
Lemma \ref{order prime to p} implies that $\Zp[K]\otimes I_{H}=\Zp[K]\otimes I_{H}^t$.
Since $\alpha_K \in  I_{K}^t$, we have $\alpha \in I_{G}^t$.
\end{proof}

\begin{lem}\label{p-prime mod p}Under the notation as in Lemma \ref{p-prime}, we suppose that  $\alpha \in  I_{G}^t$. 
We denote by $\tilde{\alpha}^{(p)}$ $($resp.\ $\tilde{\alpha}_K^{(p)})$ the image of $\alpha$ $($resp.\ $\alpha_K)$ in 
$\Z/p\Z\otimes I_G^t/I_G^{t+1}$ $($resp.\ $\Z/p\Z \otimes I_{K}^t/I_{K}^{t+1})$.
If  $\tilde{\alpha}_K^{(p)} = 0$, then
 $\tilde{\alpha}^{(p)}=0$.
\end{lem}
\begin{proof}
As in the poof of Lemma \ref{p-prime}, we have
\begin{equation*}
\alpha-\alpha_K \in \Zp[K]\otimes I_{H}^t = \Zp[K]\otimes I_{H}^{t+1} \subseteq \Zp\otimes  I_{G}^{t+1}. 
\end{equation*}Then, we have $\tilde{\alpha}^{(p)}-\tilde{\alpha}_K^{(p)} =0$ in $\Z/p\Z \otimes I_{G}^t/I_{G}^{t+1}$,
where we regard $\tilde{\alpha}_K^{(p)}$ as an element of $\Z/p\Z \otimes I_{G}^t/I_{G}^{t+1}$ under the map induced by the inclusion $K \subseteq G.$
Since $\tilde{\alpha}_K^{(p)}=0$, we have $\tilde{\alpha}^{(p)}=0$.
\end{proof}

\subsubsection{Global property}
Let $R$ be a proper subring of $\Q$ and $I_G$ the augmentation ideal of $R[G]$.

\begin{lem}\label{globallocal}Let $\alpha$ be an element of $ R[G].$
For a positive integer $t,$ the following conditions are equivalent$:$
\begin{enumerate}
\item $\alpha  \in I_G^t$.
\item $\alpha  \in \Zp \otimes_R I_G^t$ for all the primes $p$ not invertible in $R$.
\end{enumerate}
\end{lem}
\begin{proof}
This is \cite[Lemma 3.2]{dar92}.
\end{proof}

\subsection{Construction of a system of local points}\label{local study of mazur-tate elements}
With a modification of ideas of \cite{kob}, \cite{kur02} and \cite{ots}, we construct local points of $E$ to connect  Kato's Euler system with Mazur-Tate elements.

In the rest of Section \ref{section: kato and mazur-tate}, we fix a prime $p$ such that 
\begin{enumerate}
\item $p\nmid 6N\cdot|E(\mathbb{F}_p)|\prod_{\ell|N}m_{\ell}$,
\item the Galois representation $G_{\Q} \to \mathrm{Aut}_{\Zp}(T_p(E))$ is surjective,
\end{enumerate}
For an integer $S$, let $\Q(S)$ and $\Gamma_S$ be as in Section \ref{section derivatives}.
Let  $\Oc_S$ denote the ring of integers of $\Q(S)$.
If we put $H_S = \Gal(\Q(\mu_S)/\Q(S)),$ then 
we have the canonical decomposition $G_S=H_S \times \Gamma_S.$
For a finite unramified extension $K$ of $\Qp$ and its ring $\Oc$ of integers,
 let $\sigma$ denote the arithmetic Frobenius.
We denote by $\hat{E}$ the formal group law of $E$ over $\Zp$ (associated to $\omega$) and by $\mathrm{log}_{\hat{E}}$  the logarithm of $\hat{E}$,
which induces an isomorphism $\hat{E}(\Oc) \to p\Oc$.

\begin{lem}\label{denominator}
We suppose that $K$ is a finite unramified $p$-extension of $\Qp$.
Then, we have an isomorphism  defined as
\begin{equation*}
\hat{E}(\Oc) \to \Oc; \ c\mapsto \left(1-\frac{a_p}{p}\sigma + \frac{1}{p}\sigma^2\right)\log_{\hat{E}} ( c).
\end{equation*}
\end{lem}
\begin{proof}
Since $\log_{\hat{E}}: \hat{E}(\Oc) \to p\Oc$ is an isomorphism,
it suffices to show that the map $\left(1-\frac{a_p}{p}\sigma + \frac{1}{p}\sigma^2\right): p\Oc \to \Oc$ is an isomorphism.
We put $d=[K:\Qp]$.

We first assume that $p$ is a good ordinary prime of $E$, that is, $p\nmid a_p.$
Let $\alpha \in \Zp^{\times}$ be the unit root of $X^2-a_pX+p$ and $\beta \in p\Zp$ the other root.
Then we have
\begin{gather}
\label{phifrob} \left(1-\frac{\sigma}{\alpha}\right)\left(1-\frac{\sigma}{\beta}\right)=   \left(1-\frac{a_p}{p}\sigma + \frac{1}{p}\sigma^2\right).
\end{gather}
Since $p\nmid |E(\mathbb{F}_p)|$, we have $a_p\not\equiv 1 \bmod p$, and hence $\alpha  \not\equiv 1 \bmod p$.
We note that $\alpha^d -1 \in \Zp^{\times}$, since $d$ is a power of $p$.
For $A \in \Oc$, if we put
\begin{equation*}
x_A=\frac{\alpha^d}{\alpha^{d}-1}\left(\sum_{0\le k \le d-1}\frac{A^{\sigma^k}}{\alpha^{k}}\right) \in \Oc, 
\end{equation*}
then
\begin{equation*}
\left(1-\frac{\sigma}{\alpha}\right)x_A=A.
\end{equation*}
Since $\beta \in p\Zp$, 
we see that   $y_A:=-\sum_{k\ge 1}\beta^{k} x_A^{\sigma^{-k}} \in p\Oc$
converges, and it satisfies
\begin{equation*}
\left(1-\frac{\sigma}{\beta}\right)y_A=x_A.
\end{equation*}By (\ref{phifrob}), we have
\begin{equation*}
\left(1-\frac{a_p}{p}\sigma + \frac{1}{p}\sigma^{2}\right)y_A=A.
\end{equation*}
Hence, the map $\left(1-\frac{a_p}{p}\sigma + \frac{1}{p}\sigma^2\right): p\Oc \to \Oc$ is surjective, and then it is an isomorphism.

We next assume that $p$ is a good supersingular prime of $E$. 
Since $p\ge 5,$ we have $a_p =0$. 
For $A \in \Oc,$ if we put $y_A =-\sum_{k\ge 1}(-p)^kA^{\sigma^{-2k}},$ then
\begin{equation*}
\left(1+\frac{1}{p}\sigma^2\right)y_A=A.
\end{equation*}
Hence, the map $\left(1+ \frac{1}{p}\sigma^2\right): p\Oc \to \Oc$ is surjective, and then it is an isomorphism.
\end{proof}


For an integer $S$, we have  $\Oc_S \otimes_{\Z}\Zp = \prod_{\lambda|p}\Oc_{S,\lambda},$ where $\lambda$ ranges over all the primes of $\Q(S)$ dividing $p$, and $\Oc_{S,\lambda}$ denotes the completion of $\Oc_S$ at $\lambda.$
Then, the logarithm  $\log_{\hat{E}}$ naturally induces an isomorphism $\hat{E}(\Oc_S \otimes \Zp) \to p\Oc_S\otimes \Zp$,
where $\hat{E}(\Oc_S\otimes\Zp)=\oplus_{\lambda|p}\hat{E}(\Oc_{S,\lambda}).$

\begin{dfn}\label{def of c}For a square-free integer $S$ relatively prime to $p$,
we define an element $c_S\in \hat{E} \left(\Oc_S\otimes\Zp\right)$ by
\begin{equation}\label{gm}
\left(1-\frac{a_p}{p}\sigma + \frac{1}{p}\sigma^2\right)\log_{\hat{E}}(c_{S})=\mathrm{tr}_{\Q(\mu_{S})/\Q(S)}\left(\zeta_S \right) \in \Oc_S\otimes\Zp.
\end{equation}By Lemma \ref{denominator}, the element $c_S$ is well-defined.
\end{dfn}

\begin{prop}\label{norm of c_S} Let  $\ell$ be a prime not dividing $pS$.
Then, we have
\begin{equation*}
\mathrm{Tr}_{S\ell/S}(c_{S\ell})=-c_{S}^{\Fr_{\ell}^{-1}},
\end{equation*}where $\mathrm{Tr}_{S\ell/S}:\hat{E}(\Oc_{S\ell}\otimes\Zp) \to \hat{E}(\Oc_S\otimes\Zp)$ is the trace map with respect to the addition of $\hat{E}$.
\end{prop}
\begin{proof}
By Lemma \ref{denominator}, it suffices to show that 
\begin{equation*}
\mathrm{tr}_{\Q(S\ell)/\Q(S)} \circ \mathrm{tr}_{\Q(\mu_{S\ell})/\Q(S\ell)}\left(\zeta_{S\ell}\right) = - \mathrm{tr}_{\Q(\mu_{S})/\Q(S)}\left(\zeta_S^{\mathrm{Fr}_{\ell}^{-1}}\right).
\end{equation*}
Since $\mathrm{tr}_{\Q(S\ell)/\Q(S)} \circ \mathrm{tr}_{\Q(\mu_{S\ell})/\Q(S\ell)}=\mathrm{tr}_{\Q(\mu_S)/\Q(S)}\circ \mathrm{tr}_{\Q(\mu_{S\ell})/\Q(\mu_S)}$,
 we are reduced to showing that $\mathrm{tr}_{\Q(\mu_{S\ell})/\Q(\mu_S)}(\zeta_{S\ell})=-\zeta_S^{\mathrm{Fr}_{\ell}^{-1}}$,
which is not difficult to show.
\end{proof}

\begin{prop}\label{dirichlet}Let $\chi$ be a character of $\Gamma_S$.
Then, we have
\begin{gather*}
\left(1-\frac{a_p}{p}\chi(p)^{-1} + \frac{1}{p}\chi(p)^{-2}\right)\sum_{\delta\in \Gamma_S}\log_{\hat{E}}(c_S^{\delta})\chi(\delta) = \tau_S(\chi).
\end{gather*}On the right hand side, we regard $\chi$ as a character of $G_S=\Gamma_S \times H_S$ by $\chi|_{H_S}=1$.
\end{prop}

\begin{proof}
By (\ref{gm}), we have
\begin{align}\label{cs and tau}
\sum_{\delta\in \Gamma_S}\delta\left(\left(1-\frac{a_p}{p}\sigma + \frac{1}{p}\sigma^2\right)\log_{\hat{E}}(c_{S})\right) \chi(\delta)
\notag&=\sum_{\delta\in \Gamma_S}\delta\left(\mathrm{tr}_{\Q(\mu_{S})/\Q(S)}(\zeta_S)\right)\chi(\delta)\\
&=\sum_{\gamma\in G_S}\zeta_S^{\gamma}\chi(\gamma)=\tau_S(\chi).
\end{align}
On the other hand, we have
\begin{align}\label{separated computation}
\notag& \ \ \ \ \sum_{\delta\in \Gamma_S}\delta\left(\left(1-\frac{a_p}{p}\sigma + \frac{1}{p}\sigma^2\right)\log_{\hat{E}}(c_{S})\right) \chi(\delta) \\
\notag&=\sum_{\delta}\log_{\hat{E}}( c_{S}^{\delta}) \chi(\delta) - \frac{a_p}{p}\sum_{\delta}\log_{\hat{E}}( c_{S}^{\sigma\delta})\chi(\delta) + \frac{1}{p}\sum_{\delta}\log_{\hat{E}}( c_{S}^{\sigma^2\delta}) \chi(\delta)\\
\notag&=\sum_{\delta}\log_{\hat{E}}( c_{S}^{\delta}) \chi(\delta) 
- \frac{a_p}{p}\sum_{\delta}\log_{\hat{E}}( c_{S}^{\delta})\chi(\sigma^{-1}\delta) + 
\frac{1}{p}\sum_{\delta}\log_{\hat{E}}( c_{S}^{\delta}) \chi(\sigma^{-2}\delta)\\
\notag&\overset{\mathrm{(a)}}{=} \sum_{\delta}\log_{\hat{E}}(c_{S}^{\delta}) \chi(\delta) - \frac{a_p}{p}\chi(p)^{-1}\sum_{\delta}\log_{\hat{E}}( c_{S}^{\delta})\chi(\delta) + \frac{1}{p}\chi(p)^{-2}\sum_{\delta}\log_{\hat{E}}( c_{S}^{\delta}) \chi(\delta) \\
&=\left(1-\frac{a_p}{p}\chi(p)^{-1} + \frac{1}{p}\chi(p)^{-2}\right)\sum_{\delta} \log_{\hat{E}}(c_S^{\delta})\chi(\delta),
\end{align}where the equality (a) follows from $\chi(\sigma) =\chi(p)$.
Combining (\ref{cs and tau}) and (\ref{separated computation}), we complete the proof.
\end{proof}

\subsection{Kato's Euler system}\label{subsection: kato euler}
We put $T=T_p(E)$ and $V=T\otimes \Qp.$
For an integer $S$,  we have the pairing $(-,-)$ induced by the cup product
\begin{equation*}
(-,-):H^1_{f}(\Q(S)\otimes\Qp,V) \times H^1_{/f}(\Q(S)\otimes\Qp,V) \to \oplus_{\lambda|S}\Qp \to \Qp,
\end{equation*}where the last map is given by $(a_{\lambda})_{\lambda}\mapsto \sum_{\lambda}a_{\lambda}.$
Then, we have a $\Qp$-linear map
\begin{equation}\label{dual of h1f of qsf}
H^1(\Q(S)\otimes\Qp,V)  \to \mathrm{Hom}_{\Qp}\left( H^1_{f}(\Q(S)\otimes\Qp,V),\Qp\right).
\end{equation}
The exponential map $\exp_{\hat{E}}$ of $\hat{E}$ induces an isomorphism  $\Q(S)\otimes_{\Q} \Qp \to H^1_f(\Q(S)\otimes\Qp,V).$
By  taking the dual,
we have a $\Qp$-linear map
\begin{equation}\label{rough dual exp}
\mathrm{Hom}_{\Qp}\left(H^1_f(\Q(S)\otimes\Qp,V),\Qp\right) \to \Hom_{\Qp}(\Q(S)\otimes\Qp,\Qp) \cong \Q(S)\otimes\Qp,
\end{equation}where the last isomorphism comes from the perfect pairing $(\Q(S)\otimes\Qp) \times (\Q(S)\otimes\Qp)  \to \Qp$ given by $(x,y) \mapsto \mathrm{tr}_{\Q(S)/\Q}(xy).$  
The dual exponential map $\exp^*_S$  (associated to $\omega$) is defined as the composite of (\ref{dual of h1f of qsf}) and (\ref{rough dual exp})
\begin{equation*}
H^1(\Q(S)\otimes\Qp,V) \to \mathrm{Hom}_{\Qp}\left(H^1_f(\Q(S)\otimes\Qp,V),\Qp\right) \to \Q(S)\otimes\Qp.
\end{equation*}
We note that for $c \in \hat{E}(\Oc_S \otimes\Zp)$ and $z \in H^1(\Q(S)\otimes\Qp,T),$
\begin{equation}\label{dual and exp}
(c,z)=\mathrm{tr}_{\Q(S)/\Q}\left(\log_{\hat{E}}(c)\cdot \exp_{S}^{*}(z)\right) \in \Zp.
\end{equation}

Let $\mathscr{N}$ be the set of square-free products of primes relatively prime to $pN$.
By applying \cite[Lemma 9.6.1]{rub} to Kato's Euler system (cf.\ \cite[Theorems 9.7 and 12.5]{kat}), we have the following.
\begin{thm}[Kato]\label{original Kato}There exists a system $\{\z_{m}\}_{m> 0} \in \prod_m H^1(\Q(m),T)$ satisfying the following conditions.
\begin{enumerate}
\item For $m>0$ and a prime $\ell$, we have
\begin{equation*}
\mathrm{Cor}_{m\ell /m}(\z_{m\ell})=\begin{cases}
P_{\ell}(\Fr_{\ell}^{-1})\z_{m}&\mathrm{if} \ \ell\nmid pm\\
\z_{m}&\mathrm{if} \ \ell|pm.
\end{cases}
\end{equation*}
In particular $\{\z_{Sp^n}\}_{S\in \mathscr{N},n\ge 0}$ is an Euler system in the sense of Definition \ref{def of euler system}.
\item
For every character $\chi$ of $\Gamma_{m}$ of conductor $m$, we have
\begin{equation*}
\sum_{\gamma\in\Gamma_{m}}\chi(\gamma)\exp^{*}_{m}(\z_{m}^{\gamma})=\left(1-\frac{a_p\chi(p)}{p}+\frac{\chi^2(p)}{p}\right)\frac{L(E,\chi,1)}{\Omega^{+}}.
\end{equation*}
\end{enumerate}
\end{thm}
For a square-free integer $S$ relatively prime to $p$, we put 
\begin{equation*}
\Theta_S=\sum_{\gamma \in \Gamma_S}\z_S^{\gamma^{-1}}\otimes \gamma \in H^1(\Q(S),T)\otimes \Zp[\Gamma_S].
\end{equation*}
\paragraph{\textbf{Notation}}
In the rest of Section \ref{section: kato and mazur-tate}, for a finite abelian group $G$, we denote by $I_G$ the augmentation ideal of $\Zp[G].$
For an integer $S$, we denote by $\mathrm{sp}(S)$ the number of split multiplicative primes dividing $S$, and we also denote by $b_2(S)$ the number of good primes $\ell$ dividing $S$ such that  $a_{\ell}=2$ $($i.e. $P_{\ell}(1)=0).$
\begin{prop}\label{trivial zeros of euler system}
For a square-free integer $S$ relatively prime to $p$, we have
\begin{equation*}
\Theta_S \in H^1(\Q(S),T)\otimes I_{\Gamma_S}^{\mathrm{sp}(S)+b_2(S)}.
\end{equation*}
\end{prop}
\begin{proof}Our proof is similar to that of \cite[Theorem 4.2]{dar95}.
We denote by $S_1$ the product of primes $\ell$ dividing $S$ such that $\ell$ is either a split multiplicative prime or a good prime with $a_{\ell}=2$.
We prove the proposition by induction on the number $a$ of primes dividing $S_1$.
The case where $a=0$ is trivial.
We assume that $a\ge 1$ and write $S_1=\ell_1 \cdots \ell_a, S^{\prime}=S/S_1.$  
By using $\Gamma_S = \Gamma_{\ell_1} \times \cdots \times \Gamma_{\ell_a}\times \Gamma_{S^{\prime}}$, for  $\gamma \in \Gamma_S$ we write
$\gamma=\gamma_{\ell_1} \cdots \gamma_{\ell_a} \gamma^{\prime}$,
where each $\gamma_{\ell_i}$ is an element of  $ \Gamma_{\ell_i}$, and $\gamma^{\prime}\in \Gamma_{S^{\prime}}.$
Let $\mu(\cdot)$ denote the Mobius function.
Then,
\begin{align*}
&\sum_{\gamma \in \Gamma_S}\z_{S}^{\gamma^{-1}} \otimes (\gamma_{\ell_1}-1)\cdots(\gamma_{\ell_a}-1)\gamma^{\prime}
=\Theta_S +\sum_{\gamma \in \Gamma_S}\sum_{d|S_1, d\not=S_1}\mu(S_1/d)\gamma^{-1}\z_S\otimes\left(\gamma^{\prime}\prod_{\ell_i |d}\gamma_{\ell_i}\right)\\
=&\Theta_S +\sum_{d|S_1, d\not=S_1}\mu(S_1/d) \sum_{\gamma_0\in \Gamma_{dS^{\prime}}}
\gamma_0^{-1}\mathrm{Cor}_{S/dS^{\prime}}\z_S\otimes \gamma_0\\
=&\Theta_S +\sum_{d|S_1, d\not=S_1}\mu(S_1/d) \sum_{\gamma_0\in \Gamma_{dS^{\prime}}}
\gamma_0^{-1}\prod_{\ell|\frac{S_1}{d}}P_{\ell}(\Fr_{\ell}^{-1})\z_{dS^{\prime}}\otimes \gamma_0\\
=&\Theta_S +\sum_{d|S_1, d\not=S_1}\mu(S_1/d) \left(\prod_{\ell|\frac{S_1}{d}}P_{\ell}(\Fr_{\ell}^{-1})\right)\Theta_{dS^{\prime}}.
\end{align*}
Hence, we have
\begin{equation}\label{big theta}
\Theta_S =\sum_{\gamma \in \Gamma_S}\z_{S}^{\gamma^{-1}} \otimes (\gamma_{1}-1)\cdots(\gamma_a-1)\gamma^{\prime}-\sum_{d|S_1, d\not=S_1}\mu(S_1/d) \left(\prod_{\ell|\frac{S_1}{d}}P_{\ell}(\Fr_{\ell}^{-1})\right)\Theta_{dS^{\prime}}.
\end{equation}
By assumption, for each prime $\ell$ dividing $S_1/d,$ we have $P_{\ell}(1)=0$, and then $P_{\ell}(\Fr_{\ell}^{-1}) \in I_S.$
Hence, by the induction hypothesis and (\ref{big theta}), we complete the proof.
\end{proof}
\subsection{Kato's Euler system and Mazur-Tate elements}\label{subsection: kato and mazur-tate}
\begin{dfn}\label{m-t-k}For a square-free integer $S$  relatively prime to $p$, 
we define $\vartheta(\z_S)$ by
\begin{equation*}
\vartheta(\z_S)  = \sum_{\gamma\in\Gamma_S}(c_S,\z_S^{\gamma^{-1}})\gamma \in \Zp[\Gamma_S].
\end{equation*}
\end{dfn}
By abuse of notation, we denote by $\pi_{m/n}$ the natural map $\Zp[\Gamma_m] \to \Zp[\Gamma_n]$ for $n|m.$
\begin{prop}For a square-free integer $S$ relatively prime to $p$, we have the following.
\begin{enumerate}
\item Let $\ell$ be a prime not dividing $pS$. 
Then we have
\begin{equation*}
\pi_{S\ell/S}(\vartheta(\z_{S\ell}))=-\Fr_{\ell}P_{\ell}(\Fr_{\ell}^{-1})\vartheta(\z_S).
\end{equation*}
\item For every character $\chi$ of $\Gamma_S$ of conductor $S$, we have
\begin{equation*}
\chi(\vartheta(\z_S))=\tau_S(\chi)\frac{L(E,\chi^{-1},1)}{\Omega^{+}}.
\end{equation*}
\end{enumerate}
\end{prop}
\begin{proof}
By (\ref{dual and exp}),  we have
\begin{align}\label{thetass}
\vartheta(\z_S) &=\sum_{\gamma\in\Gamma_S}(c_S,\z_S^{\gamma^{-1}})\gamma \notag
=\sum_{\gamma\in \Gamma_S}\mathrm{tr}_{\Q(S)/\Q}\left(\log_{\hat{E}}(c_S)\exp^{*}_S(\z_S^{\gamma^{-1}})\right)\gamma \\ \notag
&=\sum_{\gamma\in \Gamma_S}\sum_{\delta\in \Gamma_S}\log_{\hat{E}}(c_S^{\delta})\exp^{*}_S(\z_S^{\delta\gamma^{-1}})\gamma
=\sum_{\gamma\in\Gamma_S}\sum_{\delta\in \Gamma_S}\log_{\hat{E}}(c_S^{\delta})\exp^{*}_S(\z_S^{\gamma^{-1}})\delta\gamma\\ 
&=\left(\sum_{\delta\in\Gamma_S}\log_{\hat{E}}(c_S^{\delta})\delta\right) \times \left(\sum_{\gamma\in\Gamma_S} \exp_{S}^{*}(\z_S^{\gamma^{-1}})\gamma\right) \ \ \text{in}\ (\Q(S)\otimes\Qp)[\Gamma_S].
\end{align}
By using Proposition \ref{norm of c_S}, we have
\begin{align}\label{css}\notag
&\pi_{S\ell/S}\left(\sum_{\delta\in \Gamma_{S\ell}}\log_{\hat{E}}(c_{S\ell}^{\delta})\delta\right)= 
\sum_{\delta\in\Gamma_S}\left(\mathrm{tr}_{\Q(S\ell)/\Q(S)}\left(\log_{\hat{E}}c_{S\ell}\right)\right)^{\delta}\delta
\\  \notag
=&-\sum_{\delta\in\Gamma_S}\log_{\hat{E}}\left(c_{S}^{\Fr_{\ell}^{-1}\delta}\right)\delta \\
=&-\sum_{\delta\in\Gamma_S}\log_{\hat{E}}\left(c_S^{\delta}\right)(\delta\Fr_{\ell}).
\end{align}
By Theorem \ref{original Kato}, we  have
\begin{equation}\label{katoss}
\pi_{S\ell/S}\left(\sum_{\gamma\in \Gamma_{S\ell}}\exp_{S\ell}^{*}(\z_{S\ell}^{\gamma^{-1}})\gamma\right) 
= \sum_{\gamma\in \Gamma_{S}}\exp_{S}^{*}(\z_S^{\gamma^{-1}})\gamma P_{\ell}(\Fr_{\ell}^{-1}).
\end{equation}
By (\ref{thetass}) (replacing $S$ by $S\ell$), (\ref{css}) and (\ref{katoss}), 
we deduce the assertion (1).
By (\ref{thetass}), Proposition \ref{dirichlet} and Theorem \ref{original Kato}, we conclude (2).
\end{proof}

We denote by $\theta_{S,p} \in \Zp[\Gamma_S]$ the image of the Mazur-Tate element $\theta_S$ under the natural projection $\Zp[G_S] \to \Zp[\Gamma_S]$.
\begin{cor}\label{theta}For a square-free positive integer $S$  relatively prime to $p$,
 we have 
 \begin{equation*}
\vartheta(\z_S) = \theta_{S,p}\in \Zp[\Gamma_S], 
\end{equation*}
\end{cor}
\begin{proof}
Combining the proposition above with Proposition \ref{mazur-tate}, we have 
\begin{equation*}
\chi(\theta_{S,p})=\chi(\vartheta(\z_S)) \ \ \ \text{for all the characters }\ \chi \ \text{of} \ \Gamma_S, 
\end{equation*}which shows that the element $\theta_{S,p}-\vartheta(\z_S) \in \Qp[\Gamma_{S}]$ belongs to all the  maximal ideals of $\Qp[\Gamma_{S}].$
Hence, we have $\theta_{S,p}=\vartheta(\z_S).$
\end{proof}

As in Section \ref{main chapter}, 
we put
$\mathfrak{r}_{\min}=\min_{n\ge1}\{r_{p^n}\left(H^1_{f,p}(\Q,E[p^n])\right)\}$.
\begin{cor}\label{p-vanishing wo parity}
Let $S$ be a square-free product of primes $\ell$ relatively prime to $N$ such that $E(\mathbb{F}_{\ell})[p]$ is cyclic, that is, $E(\mathbb{F}_{\ell})[p]$ is  isomorphic to $\Z/p\Z$ or $0.$
Then,
we have
\begin{equation*}
\theta_S \in I_{G_S}^{\min\{\mathfrak{r}_{\min},\ p\}}.
\end{equation*}
\end{cor}
\begin{proof}We first assume that $(S,p)=1$. 
By Lemma \ref{p-prime}, we are reduced  to proving  that
$\theta_{S,p} \in I_{\Gamma_S}^{\min\{\mathfrak{r}_{\min},\ p\}}$, which follows from
Corollaries \ref{general divisibility} and   \ref{theta}.

Next, we assume that $(S,p)\not=1$. 
If we put $S^{\prime}=S/p,$ then $p\nmid S^{\prime}.$
By the case above and Proposition \ref{mazur-tate}, we have
\begin{equation*}
\pi_{S/S^{\prime}}(\theta_{S})=-\Fr_{p}(1-a_p\Fr_{p}^{-1}+\Fr_{p}^{-2})\theta_{S^{\prime}} \in I_{G_{S^{\prime}}}^{\min\{\mathfrak{r}_{\min},\ p\}}.
\end{equation*}Since $p\nmid |G_p|,$  Lemma \ref{p-prime} implies that $\theta_{S} \in I_{G_{S}}^{\min\{\mathfrak{r}_{\min},\ p\}}.$ 
\end{proof}

We define $\displaystyle \Sel(\Q,E[p^\infty])= \varinjlim_{n}\Sel(\Q,E[p^n])$ and put $r_{p^\infty}= \mathrm{corank}_{\Zp}(\Sel(\Q,E[p^\infty]))$.
Since $\Sel(\Q,E[p^n]) \to \Sel(\Q,E[p^\infty])[p^n]$ is surjective (cf.\ \cite[Lemma 1.5.4]{rub}),
we have
\begin{equation}\label{selsel}
r_{p^n}(\Sel(\Q,E[p^n])) \ge r_{p^n}\left(\Sel(\Q,E[p^\infty])[p^n]\right) \ge r_{p{^\infty}}
\end{equation}for $n\ge 1$.
Since $p\nmid |E(\mathbb{F}_{p})|$, by the exact sequence
$0\to \hat{E}(\Zp) \to E(\Qp) \to E(\mathbb{F}_p) \to 0,$
 we have $E(\Qp)/p \cong \Z/p\Z$.
By Lemma \ref{r},
we have
\begin{equation}\label{stst}
r_{p^n}(H^1_{f,p}(\Q,E[p^n])) \ge r_{p^n}(\Sel(\Q,E[p^n]))-1\ \ \ \text{for}\    n \ge 1.
\end{equation}
Combining (\ref{selsel}) and (\ref{stst}), we have
$\mathfrak{r}_{\min} \ge r_{p^{\infty}}-1,$
and hence by Corollary \ref{p-vanishing wo parity} 
we have the following corollary.
\begin{cor}\label{vanishing with selmer rank minus 1}Let $S$ be a square-free product of  primes $\ell$ relatively prime to $N$ such that $E(\mathbb{F}_{\ell})[p]$ is cyclic. 
Then, when $r_{p^\infty} \ge 1,$ we have
\begin{equation*}
\theta_S \in I_{G_S}^{\min\{r_{p^{\infty}}-1, p\}}.
\end{equation*}
\end{cor}

\subsection{Application of the $p$-parity conjecture}\label{section of parity}
First, following \cite[Chapter 1, \S 6]{m-t},
 we recall the functional equation of Mazur-Tate elements.
Let $w_N$ be the operator on $S_2(\Gamma_0(N))$ defined as $g \mapsto \frac{1}{N\tau^2}g\left(\frac{-1}{N\tau}\right)$. 
Let $f$ be  the newform corresponding to $E$. Then there exists  $\varepsilon_f \in \{\pm 1\}$ such that
$w_N(f) = -\varepsilon_f f$. 
It is known that 
\begin{equation}\label{parity of L}
\varepsilon_f = (-1)^{\mathrm{ord}_{s=1}(L(E,s))}.
\end{equation}
Let $S$ be a positive integer relatively prime to $N$.
By  \cite[Chapter 1, \S 6]{m-t-t} and (\ref{known terminology}), for an integer $a$ relatively prime to $S$, we have
\begin{equation}\label{aprime}
\left[a/S\right]^{\pm}_E=\varepsilon_f \left[a^{\prime}/S\right]^{\pm}_E,
\end{equation}where $a^{\prime}$ is an integer  satisfying $a^{\prime}aN \equiv -1 \mod S.$
 Let $\iota$ be the map $\Q[G_S] \to \Q[G_S]$ sending $\sigma \in G_S$ to $\sigma^{-1}.$
We have a functional equation of Mazur-Tate elements as follows.
\begin{prop}\label{functional equation of mt}
\begin{equation*}
\theta_S =\varepsilon_f \delta_{-N}^{-1}\iota(\theta_S).
\end{equation*}
\end{prop}
\begin{proof}We have
\begin{align*}
&\varepsilon_f \delta_{-N}^{-1}\iota(\theta_{S})\\
=&\varepsilon_f \delta_{-N}^{-1}\sum_{a\in(\Z/S\Z)^{\times}}\left(\left[\frac{a}{S}\right]^{+}_E+\left[\frac{a}{S}\right]^{-}_E\right)\delta_a^{-1}=\varepsilon_f \sum_{a\in(\Z/S\Z)^{\times}}\left(\left[\frac{a}{S}\right]^{+}_E+\left[\frac{a}{S}\right]^{-}_E\right)\delta_{(-aN)^{-1}}\\
\overset{(\mathrm{a})}{=}&\sum_{a\in(\Z/S\Z)^{\times}}\left(\left[\frac{a^{\prime}}{S}\right]^{+}_E+\left[\frac{a^{\prime}}{S}\right]^{-}_E\right)\delta_{a^{\prime}}=\theta_S,
\end{align*}where the equation (a) follows from (\ref{aprime}).
\end{proof}
For $\gamma \in G_S$, we have
\begin{equation*}
\iota(\gamma -1 ) = \gamma^{-1}-1 \equiv -\gamma^{-1}(\gamma -1) \equiv -(\gamma-1) \mod I_{G_S}^2.
\end{equation*}Then, we have $\iota=-1$ on $I_{G_S}/I_{G_S}^2$, 
and similarly $\iota=(-1)^t$ on $I_{G_S}^t/I_{G_S}^{t+1}$ for $t\ge 1$.
 
\begin{thm}\label{p-vanishing}We suppose that   $p$ does not divide $6N \cdot |E(\mathbb{F}_p)|\prod_{\ell|N}m_{\ell}$ and the Galois representation $G_{\Q}\to \mathrm{Aut}_{\Zp}(T_p(E))$ is surjective.
Let $S$ be a square-free product of good primes $\ell$  such that $E(\mathbb{F}_{\ell})[p]$ is isomorphic to $\Z/p\Z$ or $0$.
Then, we have
\begin{equation*}
\theta_{S} \in I_{G_S}^{\min\{r_{p^\infty}, p\}} \subseteq \Zp[G_S].
\end{equation*}In particular, if $p\ge r_E$ then $\theta_S\in I_{G_S}^{r_E}.$
\end{thm}
\begin{proof}By Corollary \ref{vanishing with selmer rank minus 1}, we have $\theta_{S} \in I_{G_S}^{\min\{r_{p^\infty}-1, p\}}.$ 
If $p\le r_{p^{\infty}}-1$ or $r_{p^{\infty}}=0$, then there is nothing to prove. 
Hence, we assume that $1 \le r_{p^{\infty}} \le p$, 
and then $\theta_S \in I_{G_S}^{r_{p^{\infty}}-1}$.
We note that the group $ G_S$ acts on $I_{G_{S}}^{r_{p^{\infty}}-1}/I_{G_S}^{r_{p^{\infty}}}$ trivially. 
Then, modulo $I_{G_S}^{r_{p^{\infty}}}$, we have
\begin{equation}\label{uuii}
\theta_S \equiv \varepsilon_f \delta_{-N}^{-1}\iota(\theta_S) \equiv \varepsilon_f \delta_{-N}^{-1}(-1)^{r_{p^{\infty}}-1}\theta_S\equiv \varepsilon_f (-1)^{r_{p^{\infty}}-1}\theta_S.
\end{equation}
By the $p$-parity conjecture (cf.\ \cite[Theorem 1.4]{dok}),
\begin{equation*}
(-1)^{\mathrm{ord}_{s=1}(L(E,s))} = (-1)^{r_{p^{\infty}}}.
\end{equation*}
Combining this with (\ref{parity of L}) and  (\ref{uuii}),
we have
\begin{equation*}
2\theta_S \equiv 0 \mod I_{G_S}^{r_{p^{\infty}}}.
\end{equation*}Since $2$ is assumed to be invertible in $\Zp$, we conclude that 
$\theta_S \in I_{G_S}^{r_{p^{\infty}}}$.
\end{proof}

\section{Proof of the main result}\label{relation}
\subsection{The order of vanishing}\label{the proof of the main result}
Let $R$ be a subring of $\Q$ in which the primes satisfying  at least one of the following conditions are invertible: 
\begin{enumerate}
\item $p$ divides  $6N\cdot|E(\mathbb{F}_p)|\prod_{\ell|N}m_{\ell}$,
\item the Galois representation $G_{\Q} \to \mathrm{Aut}_{\Zp}(T_p(E))$ is \textit{not} surjective,
\item $p < r_E$.
\end{enumerate}
For an integer $S$,
we denote by $I_S$ the augmentation ideal of $R[G_S]$.

\begin{thm}\label{weak vanishing thm}Let $S$ be a square-free product of good primes $\ell$ such that for each prime $p$ not invertible in $R$, the module $E(\mathbb{F}_{\ell})[p]$ is cyclic. Then, we have
\begin{equation*}
\theta_S \in I_S^{r_E}.
\end{equation*}
\end{thm}

\begin{proof}
For a prime $p$ not invertible in $R$, we see that $S$ satisfies the assumption of Theorem \ref{p-vanishing},
and then $\theta_S \in \Zp \otimes_R I_S^{r_E}$.
By Lemma \ref{globallocal}, we complete the proof.
\end{proof}

\begin{rem}\label{density computation}
For distinct primes $p$ and $\ell$, the module $E(\mathbb{F}_{\ell})[p]$ is  isomorphic to $0,  \Z/p\Z$ or $(\Z/p\Z)^{\oplus 2}$. 
Furthermore, $E(\mathbb{F}_{\ell})[p] \cong (\Z/p\Z)^{\oplus 2}$ if and only if $\ell$ splits completely in $\Q(E[p]).$
By Chebotarev's density theorem, if the representation $G_{\Q} \to \mathrm{Aut}(E[p])$ is surjective, then the density of such primes $\ell$  is equal to $1/|\mathrm{GL}_2(\mathbb{F}_p)|=1/(p^2-1)(p^2-p).$
Hence, if we denote by $\delta_R$ the density (if it exists)  of primes $\ell$ that satisfy the assumption of Theorem \ref{weak vanishing thm}, then
\begin{equation*}
\delta_R \ge 1-\sum_{p\notin R^{\times}}\frac{1}{(p^2-1)(p^2-p)} \ge 1-\sum_{p\ge 5}\frac{1}{(p^2-1)(p^2-p)}> 0.99.
\end{equation*}
\end{rem}
\begin{cor}\label{under additional assumption}We assume that $\prod_{\ell^{\prime}}(\ell^{\prime}-1) \in R^{\times}$, where $\ell^{\prime}$ ranges over all the split multiplicative primes.
Let $S$ be a square-free product of primes $\ell$ as in Theorem \ref{weak vanishing thm} and  split multiplicative primes.
If $b_{2,R}(S)$ denotes the number of primes $\ell$ dividing $S$ such that  $\ell-1 \in R^{\times}$ and $a_{\ell}=2,$
then
\begin{equation*}
\theta_{S} \in I_{S}^{r_E+\mathrm{sp}(S)+b_{2,R}(S)}.
\end{equation*}
\end{cor}
\begin{proof}
By Lemma \ref{globallocal}, it suffices to show that for each prime $p$ not invertible in $R$,
\begin{equation}\label{local refined conjecture}
\theta_{S} \in \Zp \otimes I_{S}^{r_E+\mathrm{sp}(S)+b_{2,R}(S)}.
\end{equation}We denote by $S^{\prime}$ the product of the primes $\ell$ dividing $S$ such that $\ell$ is either a split multiplicative prime, or a good prime with $a_{\ell}=2$ and $\ell-1 \in R^{\times}$.
We put $S_0 =S/S^{\prime}.$
By Proposition \ref{mazur-tate} (1),
we have
\begin{equation*}
\pi_{S/S_0}(\theta_{S})=\left(\prod_{\ell|S^{\prime}}-\Fr_{\ell}P_{\ell}(\Fr_{\ell}^{-1})\right)\theta_{S_0}.
\end{equation*}We note that for each prime $\ell |S^{\prime},$ we have $P_{\ell}(1) =0,$ and then $P_{\ell}(\Fr_{\ell}^{-1}) \in I_{\ell}.$
Hence, by Theorem \ref{weak vanishing thm}, we have $\pi_{S/S_0}(\theta_{S}) \in I_{S_0}^{r_E+\mathrm{sp}(S)+b_{2,R}(S)}.$
Since each prime divisor $\ell$ of $S^{\prime}$ satisfies $\ell-1 \in R^{\times}$,
by Lemma \ref{p-prime} we obtain (\ref{local refined conjecture}).
\end{proof}

\subsection{The leading coefficients}
 Let $S$ be as in Theorem \ref{weak vanishing thm}.
Let  $p $ be a prime not invertible in $R$ such that $p\nmid S$.
As in Section \ref{section derivatives}, we denote by $\mathscr{R}_p$ the set of good primes $\ell$
 such that $\ell \equiv 1 \mod p.$
 We then write $S=\ell_1\cdots \ell_s$, where $\ell_1,\ell_2,\ldots,\ell_n \in \mathscr{R}_{p}$, and $\ell_{n+1},\ldots,\ell_s \notin \mathscr{R}_{p}$.
We put $S_1=\ell_1\cdots\ell_n$ and $S_2=\ell_{n+1}\cdots\ell_s.$
We denote by $\tilde{\theta}_{S}^{(p)}$ the image of $\theta_S$ in $\Z/p\Z \otimes_R I_S^{r_E}/I_{S}^{r_E+1}.$
\begin{thm}\label{theorem on leading term}
 If
$\tilde{\theta}_{S}^{(p)}\not=0,$
then we have
\begin{equation*}
\Sha[p]=0,  \ \ \ \ p\nmid J_{S_1},\ \ \ \ \ p\nmid \prod_{\ell|S_2}(a_{\ell}-2). 
\end{equation*}
\end{thm}
\begin{rem}
By the relation between our $\theta_S$ and the original Mazur-Tate element in \cite{m-t-t}, the same theorem for the leading coefficients considered in \cite{m-t-t} also holds (cf. Section \ref{chapter of refined BSD}).
\end{rem}
\begin{proof}Let $\theta_{S,p}, \Gamma_S$ and  $I_{\Gamma_S}$ be as in Subsection \ref{subsection: kato and mazur-tate}.
First, we assume that $r_E\ge1$. 
We denote by $\tilde{\theta}_{S,p}^{(p)}$ the image of $\theta_{S,p}$ in $\Z/p\Z \otimes I_{\Gamma_S}^{r_E}/I_{\Gamma_S}^{r_E+1}.$ 
Then by Lemma \ref{p-prime mod p}, we have
\begin{equation}\label{non-vanishing of leading term}
\tilde{\theta}_{S,p}^{(p)} \not=0.
\end{equation}
By applying Proposition \ref{taylor} to $\sum_{\gamma \in \Gamma_S}\z_{S}^{\gamma^{-1}}\otimes\gamma$ and by using Corollary \ref{theta}, we have
\begin{equation*}
\theta_{S,p} = \sum_{\underline{k}=(k_1,\ldots,k_s) \in \Z_{\ge 0}^{\oplus s}}(c_S,D_{\underline{k}}\z_S)(\sigma_{\ell_1}^{-1}-1)^{k_1}\cdots (\sigma_{\ell_s}^{-1}-1)^{k_s}.
\end{equation*}
By Lemma \ref{order prime to p},
\begin{equation*}
\theta_{S,p} \equiv \sum_{k_1+\cdots + k_n \le r_E \atop k_{n+1}=\cdots=k_s=0}(c_S,D_{\underline{k}}\z_S)(\sigma_{\ell_1}^{-1}-1)^{k_1}\cdots (\sigma_{\ell_s}^{-1}-1)^{k_s} \mod I_{\Gamma_S}^{r_E+1}.
\end{equation*}
For $\underline{k}$ such that $k_1+\cdots + k_n < r_E$ and $k_{n+1}=\cdots=k_s=0$,
we have $D_{\underline{k}}=N_{S_2} D^{\prime}$, 
where $D^{\prime}=D_{\ell_1}^{(k_1)}\cdots D_{\ell_n}^{(k_n)}$.
We put $S^{\prime}=\mathrm{Cond}(D^{\prime}).$
By applying  Lemma \ref{r} to the exact sequence
\begin{equation*}
0\to H^1_{f,S^{\prime}}(\Q,E[p])\to \Sel(\Q,E[p]) \to \oplus_{\ell|S^{\prime}}E(\Q_{\ell})/p,
\end{equation*} we have  $r_E\le r_p(H^1_{f,S^{\prime}}(\Q,E[p]))+r_p(A(S^{\prime}))$. 
By applying  Theorem \ref{s} to $D^{\prime}\z_{S_1}$, we have
\begin{align*}\label{vanishing coefficient}
\loc_{p}(D\z_S \bmod p) &=\loc_p( N_{S_2}D^{\prime}\z_S \bmod p)=\prod_{\ell|S_2}P_{\ell}(\Fr_{\ell}^{-1})\loc_p(D^{\prime}\z_{S_1} \bmod p)\\
&\in H^1_{f}(\Q(S)\otimes\Qp,E[p]).
\end{align*}
Then,  we have $(c_S,D\z_S) \equiv 0 \mod p,$ and
hence 
\begin{equation*}
\theta_{S,p} \equiv \sum_{k_1+\cdots + k_n = r_E \atop k_{n+1}=\cdots=k_s=0}(c_S,D_{\underline{k}}\z_S)(\sigma_{\ell_1}^{-1}-1)^{k_1}\cdots (\sigma_{\ell_s}^{-1}-1)^{k_s} \mod p\Zp[\Gamma_S]+ I_{\Gamma_S}^{r_E+1}.
\end{equation*}
By (\ref{non-vanishing of leading term}), there exists $\underline{k}$ such that $k_1+\cdots + k_n = r_E$, $k_{n+1}=\cdots=k_s=0$  and $(c_S,D_{\underline{k}}\z_S) \not\equiv 0 \mod p.$
For this $\underline{k}$, if we let $D^{\prime}$ be as above, then 
$\loc_p(D^{\prime}\z_{S_1} \bmod p) \notin H^1_{f}(\Q(S_1)\otimes\Qp,E[p]).$
Since $\mathrm{ord}(D^{\prime})=r_E$, by applying Corollary \ref{toward leading term} to $D^{\prime}\z_{S_1},$ we have $\Sha[p]=0$ and the map $E(\Q)/p \to \oplus_{\ell|S_1}E(\Q_{\ell})/p$ is surjective.
Since $\oplus_{\ell|S_1}E(\Q_{\ell})/p\cong \oplus_{\ell|S_1}E(\mathbb{F}_{\ell})/p$ and $p\nmid \prod_{\ell|N}m_{\ell},$ 
the map 
\begin{equation*}
E(\Q) \to \left[\left(\oplus_{\ell|S_1}E(\mathbb{F}_{\ell})\right) \oplus\left( \oplus_{\ell|N}E(\Q_{\ell})/E_0(\Q_{\ell}) \right)\right]\otimes\Z/p\Z
\end{equation*} is surjective, that is,  $p\nmid J_{S_1}.$
Since $\Gamma_{S_1}$ acts on $I_{S_1}^{r_E}/I_{S_1}^{r_E+1}$ trivially and $\theta_{S_1} \in I_S^{r_E}$, we have 
 \begin{align}\label{s1s2}
\notag  \pi_{S/S_1}(\theta_S) &\equiv\left(\prod_{\ell|S_2}-P_{\ell}(1)\right)\theta_{S_1}\\
& \equiv \left(\prod_{\ell|S_2}(a_{\ell}-2)\right)\theta_{S_1} \ \ \mod I_{S_1}^{r_E+1}.
 \end{align}
Since $p\nmid |G_{\ell}|$ for $\ell|S_2$, by Lemma \ref{p-prime mod p}, the image of $ \pi_{S/S_1}(\theta_S)$ in $\Z/p\Z \otimes I_{S_1}^{r_E}/I_{S_1}^{r_E+1}$ is not zero. 
Hence, by (\ref{s1s2}), we have $p\nmid \prod_{\ell|S_2}(a_{\ell}-2).$

We assume that $r_E=0$. 
By Proposition \ref{mazur-tate}, we have
\begin{equation}\label{theta1}
\pi_{S/1}(\theta_S) = \left(\prod_{\ell|S}(a_{\ell}-2)\right)\theta_1=\left(\prod_{\ell|S}(a_{\ell}-2)\right)\frac{L(E,1)}{\Omega^{+}}\in R.
\end{equation}
Since $\pi_{S/1}(\theta_S) \not\equiv 0 \bmod  p$ and $\theta_1 \in R$, we have $\frac{L(E,1)}{\Omega^{+}}\not\equiv 0 \bmod p.$
By the work of Kolyvagin and Kato (cf.\ \cite[Theorem 3.5.11]{rub}), we have $\Sha[p]=0.$
The equation (\ref{theta1}) also implies that 
\begin{equation}\label{saigo}
p\nmid \prod_{\ell|S}(a_{\ell}-2).
\end{equation}
Since $r_E=0$ and $E(\Q)[p]=0$, we have $E(\Q)/p=0$. 
We note that $|E(\mathbb{F}_{\ell})| \equiv 2-a_{\ell}$ for $\ell \in \mathscr{R}_{p}.$ 
Then by (\ref{saigo}), we have $p\nmid \prod_{\ell|S_1}|E(\mathbb{F}_{\ell})|$, and hence $p\nmid J_{S_1}.$
\end{proof}


\section{Trivial zeros}By an elementary argument, we show that Mazur-Tate elements have trivial zeros induced by  split multiplicative primes and by good primes for which the Hasse-invariant is equal to $2$.
In this section, the Mordell-Weil rank is not involved, and we do not require  divisibility of derivatives of Euler systems or the $p$-parity conjecture.

Let $R$ be a subring of $\Q$  in which all the primes satisfying  at least one of the  conditions (i), (ii) in Subsection \ref{subsection: main result} are invertible.

\begin{thm}\label{p-adic trivial zeros}
For  a square-free integer $S>0$ and $b_2(S)$ as above Proposition \ref{trivial zeros of euler system},  
\begin{equation*}
\theta_{S} \in I_{S}^{\mathrm{sp}(S)+b_{2}(S)}.
\end{equation*}
\end{thm}
\begin{rem}
In our setting, Theorem \ref{p-adic trivial zeros} implies Conjecture \ref{main conj intro} when $r_E=0$.
Although Bergunde-Gehrman \cite{b-g} has announced that they proved that $\theta_S \in I_S^{\mathrm{sp}(S)}$ in the general case, zeros coming from good primes $\ell$ with $a_{\ell}=2$ are not considered.
\end{rem}
\begin{proof}We put $a=\mathrm{sp}(S)+b_{2}(S).$
By Lemma \ref{globallocal}, it suffices to show that for each prime $p$ not invertible in $R$,
\begin{equation*}
\theta_{S} \in \Zp \otimes I_{S}^{a}.
\end{equation*}
If $p\nmid S$, then
by Lemma \ref{p-prime} and Corollary \ref{theta}, it suffices to show that 
$\theta(\z_S) \in \Zp \otimes I_{\Gamma_S}^a$,
which follows from Proposition \ref{trivial zeros of euler system}.
In the case where $p|S,$ we have
\begin{equation*}
\pi_{S/\frac{S}{p}}(\theta_S)=-\Fr_{p}P_{p}(\Fr_{p}^{-1})\theta_{S/p}.
\end{equation*}Since $|G_p| \in \Zp^{\times},$ by the case above and Lemma \ref{p-prime}, we complete the proof.
\end{proof}

\end{document}